\documentclass[12pt]{amsart}
\usepackage{amscd}
\usepackage{mathrsfs}
\usepackage{a4wide}
\usepackage{color}
\usepackage{amssymb,amsmath,amsthm,amsfonts,latexsym}
\usepackage[utf8x]{inputenc}
\usepackage{bbm} 

\usepackage{graphicx}

\usepackage{amsmath,amsfonts,amssymb,amsthm} 
\usepackage{tikz}
\usepackage{hyperref}
\usepackage[all]{xy}

\newtheorem{theorem}{Theorem}[section]

\newtheorem{remark}[theorem]{Remark}
\newtheorem{claim}[theorem]{Claim}

\newtheorem{proposition}[theorem]{Proposition}
\newtheorem{example}[theorem]{Example}
\newtheorem{definition}[theorem]{Definition}
\newtheorem{question}[theorem]{Question}

\newcommand{\comm}[1]{}

\newcommand{\dom}{\operatorname{dom}}

\def\leukfrac#1/#2{\leavevmode
               \kern.1em
                \raise.9ex\hbox{\the\scriptfont0 ${}_#1$}
                \hskip -1pt\kern-.1em
                /\kern-.15em\lower.10ex\hbox{\the\scriptfont0 ${}_#2$}}

\makeatletter
\def\diam{\mathop{\operator@font diam}\nolimits}
\makeatother

\usepackage[shortlabels]{enumitem}
\newcommand{\comment}[1]{$\blacktriangleright$\ {
		\sf #1}\ 
$\blacktriangleleft$}
\newcommand{\hide}[1]{}

\newcommand{\GG}{{\mathcal G}}
\usepackage[nobysame,initials]{amsrefs}

\title{Measurable versions of Vizing's theorem}
\author{Jan Greb\'ik}
\author{Oleg Pikhurko}
\address{
Mathematics Institute.
University of Warwick,
Coventry CV4~7AL, UK} 
\email{jan.grebik@warwick.ac.uk}
\address{
Mathematics Institute and DIMAP.
University of Warwick,
Coventry CV4~7AL, UK}
\email{O.Pikhurko@warwick.ac.uk}
\thanks{\emph{Jan Grebík} was supported by the GACR project GJ16-07822Y, RVO:67985807, and
Leverhulme Research Project Grant RPG-2018-424.
\emph{Oleg Pikhurko} was supported by Leverhulme Research Project Grant RPG-2018-424.}

\begin{document}

\begin{abstract}
We establish two versions of Vizing's theorem for Borel multi-graphs whose vertex degrees and edge multiplicities are uniformly bounded by respectively $\Delta$ and $\pi$. The ``approximate'' version states that, for any Borel probability measure on the edge set and any $\epsilon>0$, we can properly colour  all but $\epsilon\textsl{}$-fraction of edges with $\Delta+\pi$ colours in a Borel way.
The ``measurable'' version, which is our main result, states that if, additionally, the measure is invariant, then there is a measurable proper edge colouring of the whole edge set with at most $\Delta+\pi$ colours.
\end{abstract}

\maketitle

\section{Introduction}\label{sec:Introduction}

One of fundamental notions of graph theory is the \emph{chromatic index $\chi'(G)$} of a graph $G$ which is the smallest number of colours needed to colour all edges of $G$ so that every two edges that intersect have different colours. This definition also applies to \emph{multi-graphs} (where a pair of vertices may be connected by more than one edge, with the edges supported on the same pair of vertices called \emph{parallel}); in particular, no two parallel edges are allowed to have the same colour.
\comm{
While the chromatic index is a special case of the chromatic number (namely, $\chi'(G)$ is equal to $\chi(L(G))$, the chromatic number of the line graph of $G$),
the theory of edge-colourings in, a sense, more satisfactory, not least because of its direct connections to matching theory.
}

For a multi-graph $G$ (finite or infinite), let $\Delta(G)$ denote the \emph{maximum degree}, that is, the maximum number of edges that are incident to a vertex. Also, let $\pi(G)$ be the \emph{maximum multiplicity} of $G$, that is, the maximum number of edges with the same endpoints. Thus $\pi(G)= 1$ for graphs with non-empty edge sets.
The multi-graphs that we consider in this paper (discrete or Borel) are always of bounded degree and therefore $\Delta(G),\pi(G)\in \mathbb{N}$.

The \emph{greedy upper bound} $\chi'(G)\le 2\Delta(G)-1$ for a finite multi-graph $G$ can be established by a simple greedy algorithm that colours edges one by one. (Note that at most $2(\Delta(G)-1)$ colours can be forbidden at any edge that we are about to colour.) Shannon~\cite{Shannon49} proved that $\chi'(G)\le \lfloor\, \frac32\,\Delta(G)\,\rfloor$ and a simple example shows that this bound is best possible as a function of the maximum degree only. A remarkable theorem of Vizing~\cite{Vizing64}, that was also proved independently by Gupta~\cite{Gupta66}, states that $\chi'(G)\le \Delta(G)+\pi(G)$.
In particular, if $G$ is a graph then $\chi'(G)\le \Delta(G)+1$, which is best possible when $\Delta(G)\ge 2$.
(Incidentally, let us observe that the best possible upper bound on $\chi'(G)$ as a function of $\Delta(G)$ and $\pi(G)$ is not known in general, see 
Scheide and Stiebitz~\cite{ScheideStiebitz12} for our current knowledge on this question.) 
 Also, the much earlier theorem of K\H onig \cite{Konig16} states that $\chi'(G)\le \Delta(G)$ when $G$ is a bipartite graph, and this bound extends to bipartite multi-graphs. These classical results laid the foundation of edge-colouring, an important and active area of graph theory; see, for example, 
the recent book on edge-colouring by Stiebitz, Scheide, Toft and Favrholdt~\cite{StiebitzScheideToftFavrholdt:gec}.

In this paper, we consider mostly infinite (multi-)graphs but, as was mentioned earlier, we restrict ourselves only to ones of bounded {maximum degree}.
If one does not impose any further structure then, for example, Vizing's theorem extends to infinite  multi-graphs by the Axiom of Choice. Indeed, every finite subgraph is edge-colourable by the original 
theorem so the Compactness Principle gives the required edge-colouring of the whole multi-graph. The focus of this paper is to find ``constructive'' edge-colourings.

Kechris, Solecki and Todorcevic~\cite{KechrisSoleckiTodorcevic99} initiated systematic study of Borel colourings. One of the basic objects here is a \emph{Borel graph}
which is a triple $\GG=(V,\mathcal B,E)$, where $(V,\mathcal B)$ is a standard Borel space and $E$ is a Borel subset of $[V]^2:=\{\,\{x,y\}: \mbox{distinct }x,y\in V\}$.
\comm{These objects appear in descriptive combinatorics (see e.g.\ the survey by Kechris and Marks~\cite{KechrisMarks:survey}) and Borel orbit equivalence (see e.g.\ the book by Kechris and Miller~\cite{KechrisMiller:toe}) and have connections to many other areas. For example, using Borel graphs Elek and Lippner~\cite{ElekLippner10} gave another proof of the result of Nguyen and Onak~\cite{NguyenOnak08} that the matching ratio in bounded-degree graphs is testable. Borel circle squaring?}
Define the \emph{Borel chromatic number} $\chi_{\mathcal B}(\GG)$ of a Borel graph $\GG$ to be the minimum $k\in\mathbb N$ such that  there is a Borel partition $V=V_1\cup\dots\cup V_k$
into \emph{independent} sets (that is, sets that span no edge of $E$). Also, the \emph{Borel chromatic index} $\chi'_{\mathcal B}(\GG)$ is  the smallest number of Borel matchings that partition~$E$. (By a 
\emph{matching} we understand a set of pairwise disjoint edges; we do not require that every vertex is covered.)
 For an illustration, consider the following example.

\begin{example}\label{ex:translation} Given $\alpha\in\mathbb R\setminus \mathbb Q$, let $\mathcal T_\alpha:=([0,1),\mathcal B,E)$ be the Borel graph on the unit real interval $[0,1)$ with edge-set $E:=\{\,\{x,x\pm\alpha\pmod 1\}:x\in [0,1)\,\}$. Also, let $\lambda$ denote the Lebesgue measure on $([0,1),\mathcal B)$.
\end{example}

This example  exhibits various interesting properties that contradict ``finite intuition''. Namely, $E$ defines a 2-regular and acyclic graph while the ergodicity of $x\mapsto x+2\alpha \pmod 1$ implies that every Borel vertex 2-colouring or every Borel matching misses a set of vertices of positive Lebesgue measure. Thus 
each of $\chi_{\mathcal B}(\mathcal T_\alpha)$ and $\chi'_{\mathcal B}(\mathcal T_\alpha)$ is strictly larger than~$\chi(\mathcal T_\alpha)=\chi'(\mathcal T_\alpha)=2$.

The following important result (which also extends to multi-graphs) shows that the upper bounds coming from simple greedy algorithms also apply in the Borel setting.

\begin{theorem}[Kechris et al~\cite{KechrisSoleckiTodorcevic99}]\label{th:KST} For every Borel graph $\GG$ of bounded maximum degree, we have that $\chi_{\mathcal B}(\GG)\le \Delta(\GG)+1$ and $\chi'_{\mathcal B}(\GG)\le  2\Delta(\GG)-1$.\qed
\end{theorem}

Remarkably, the upper bounds of Theorem~\ref{th:KST} are best possible, even if we insist that each connectivity component is a tree and, for the Borel chromatic index, that
the graph comes with a Borel bi-partition with no edge inside a part:

\begin{theorem}[Marks~\cite{Marks16}] For every $d\ge 3$, there are $d$-regular acyclic Borel graphs $\GG$ and $\mathcal{H}$ such that
	\begin{enumerate}[(i),nosep]
		\item\label{it:Marks1}
		$\chi_{\mathcal B}(\GG)=d+1$,
		\item\label{it:Marks2} $\chi_{\mathcal B}(\mathcal{H})=2$ and $\chi'_{\mathcal B}(\mathcal{H})=2d-1$.\qed
	\end{enumerate}
\end{theorem}

Note that if $d=2$, then the Borel graph from Example~\ref{ex:translation} satisfies Property~\ref{it:Marks1} while a 2-regular Borel graph satisfying~\ref{it:Marks2} was earlier constructed by Laczkovich~\cite{Laczkovich88}.

\comm{
Rather surprisingly, Marks~\cite{Marks16} constructed, for every $d\ge 3$, an example of a $d$-regular Borel graph $\GG$ such that $\GG$ has no cycles, and
\begin{enumerate}[(a),nosep]
	\item\label{it:Marks1}
	$\chi_{\mathcal B}(\GG)=d+1$, or
	\item\label{it:Marks2} $\chi_{\mathcal B}(\GG)=2$ and $\chi'_{\mathcal B}(\GG)=2d-1$.
	\end{enumerate}
 If $d=2$, then the irrational rotation of a circle gives an easy example of a graph that satisfies Property~\ref{it:Marks1} while a 2-regular graph satisfying~\ref{it:Marks2} was constructed by Laczkovich~\cite{Laczkovich88}. Thus 
 the upper bounds of Theorem~\ref{th:KST} are best possible, even if we insist that each connectivity component is a tree and, for the Borel chromatic index, that
 the graph comes with a Borel bi-partition with no edge inside a part.
}
 
In many applications of Borel graphs, one can ignore a null-set with respect to some given measure. This motivates the following notions. 

\begin{definition}
 Given a Borel graph $\GG=(V,\mathcal B,E)$ and a probability measure $\mu$ on $(V,\mathcal B)$, let
  \begin{itemize}
  	\item the \emph{$\mu$-measurable chromatic number $\chi_{\mu}(\mathcal B)$} be the smallest integer $k$ for which there is a Borel 
  		partition $V=V_0\cup V_1\cup\dots\cup V_k$ such that $V_i$ spans no edge in $\GG$ for each $i\in [k]:=\{1,\dots,k\}$
  		while $\mu(V_0)=0$;
  		\item the \emph{$\mu$-measurable chromatic
  			index} $\chi'_{\mu}(\mathcal B)$ be the smallest integer $k$ for which there is a Borel 
  		partition $E=E_0\cup E_1\cup\dots\cup E_k$ such that $E_i$ is a matching for each $i\in [k]$
  		while $V(E_0)$, the set of vertices covered by $E_0$, has measure zero in~$\mu$.
  	\end{itemize}
\end{definition}

Clearly, by allowing some ``errors'' (restricted to $V_0$ and $E_0$ in the above definitions) we get greater flexibility and thus $\chi_\mu(\GG)\le \chi_{\mathcal B}(\GG)$ and $\chi'_\mu(\GG)\le \chi'_{\mathcal B}(\GG)$ for every probability measure $\mu$ on $(V,\mathcal{B})$. Conley, Marks and Tucker-Drob~\cite[Theorem~1.2]{ConleyMarksTuckerdrob16} showed that $\chi_\mu(\GG)\le \Delta(\GG)$ for every probability measure $\mu$, provided $\Delta(\GG)\ge3$ and $\GG$ does not contain a clique on $\Delta(\GG)+1$ vertices, thus proving a measurable version of Brooks' theorem~\cite{Brooks41}. (See also Bernshteyn~\cite[Theorem~3.4]{Bernshteyn20arxiv} for a strengthening of this result.) Note that the case $\Delta(\GG)=2$ is special because of e.g.\ Example~\ref{ex:translation} with respect to  the Lebesgue measure; see Conley et al~\cite[Theorem~1.6]{ConleyMarksTuckerdrob16} for a characterization of $(\GG,\mu)$ with $\chi_\mu(\GG)>\Delta(\GG)=2$.

Marks~\cite[Question 4.9]{Marks16} asked if a measurable version of Vizing's theorem holds for arbitrary Borel probability measures:

\begin{question}[Marks~\cite{Marks16}]
\label{qu:Marks}
 Is it true that, for every Borel graph $\GG=(V,\mathcal B,E)$ of bounded maximum degree and every probability measure $\mu$ on $(V,\mathcal B)$, we have
 \begin{equation}\label{eq:MarksQn}
 \chi'_{\mu}(\GG)\le \Delta(\GG)+1?
\end{equation} 
\end{question}

Marks proved~\cite[Theorem 4.8]{Marks16} that this is the case for $\Delta(\GG)=3$. (And it is not hard to show that 
(\ref{eq:MarksQn}) holds when $\Delta(\GG)\le 2$.)

An important case is when the probability measure $\mu$ is \emph{$E$-invariant} (that is, every Borel partial injective map $\phi;V\to V$ with $(x,\phi(x))\in E$ for all $x$ in the domain of $\phi$ preserves the measure $\mu$), in which case we call
the quadruple $\GG=(V,\mathcal B,E,\mu)$ a \emph{graphing}.
For example, the quadruple $([0,1),\mathcal B, E, \lambda)$ from Example~\ref{ex:translation} is a graphing. Graphings appear in descriptive combinatorics (see e.g.\ the survey by Kechris and Marks~\cite{KechrisMarks:survey}), orbit equivalence (see e.g.\ the book by Kechris and Miller~\cite{KechrisMiller:toe}), measured group theory (see e.g.\ the surveys~\cite{Furman11,Gaboriau02,Gaboriau10,Kechris:gaega,Shalom05}), sparse graph limits (see e.g.\ the book by Lov\'asz~\cite[Part~4]{Lovasz:lngl}), and have connections to many other areas. In fact, the question whether~\eqref{eq:MarksQn} holds for every graphing $\GG=(V,\mathcal B,E,\mu)$ was earlier asked by Ab\'ert~\cite[Question~35]{Abert10questions}.

Cs\'oka, Lippner and Pikhurko~\cite[Theorem~1.5]{CsokaLippnerPikhurko16} proved that, for a graphing $\GG$ of bounded maximum degree, we have $\chi_{\mu}'(\GG)\le \Delta(\GG)+1$ when $\GG$ has no odd cycles and $\chi_{\mu}'(\GG)\le \Delta(\GG)+O(\Delta(\GG)^{1/2})$ in general. In a related result, Bernshteyn~\cite[Theorem~1.3]{Bernshteyn19am} proved that $\Delta(\GG)+o(\Delta(\GG))$ colours are enough for measurable edge-colouring (even for the so-called list-colouring version) provided that the graphing $\GG$ 
factors to the shift action $\Gamma\curvearrowright [0,1]^\Gamma$ of a finitely generated group~$\Gamma$.

Our main result is to prove the best possible bound on measurable chromatic index of a general graphing in terms of its maximal degree. If fact, our proof also works when we allow multiple edges so we present this more general case. Namely, a \emph{Borel multi-graph} with multiplicity at most $\pi\in \mathbb{N}$ 
is a triple $(V,\mathcal B,E)$ where $(V,\mathcal B)$ is a standard Borel space and $E$ is a Borel subset of $[V]^2\times [\pi]$. For $\{x,y\}\in [V]^2$, we view the pairs $(\{x,y\},k)\in E$ as the parallel edges with end-points $x$ and $y$. A \emph{multi-graphing} is a quadruple $(V,\mathcal B,E,\mu)$, where $(V,\mathcal B,E)$ is a Borel multi-graph and $\mu$ is a probability measure on $(V,\mathcal B)$ which is invariant with respect to the projection of $E$ onto $[V]^2$ (that is, when we replace all parallel edges by one edge). In the obvious way, we define the (Borel or measurable) chromatic index, etc. 
In this notation, we can prove the following measurable version of Vizing's theorem.

\begin{theorem}[Main Result]
	\label{th:main} For every Borel multi-graphing $\GG=(V,\mathcal B,E,\mu)$ with bounded maximum degree, it holds that $\chi_{\mu}'(\GG)\le \Delta(\GG)+\pi(\GG)$.
\end{theorem}


One application of measurable versions of Vizing's theorem was observed by Cs\'oka et al~\cite{CsokaLippnerPikhurko16}. Namely, consider the smallest $k=k(d)$ such that for every graphing $\GG=(V,\mathcal B,E,\mu)$ with $\Delta(\GG)=d$ there are invertible measure-preserving maps $\phi_i:A_i\to B_i$ with $A_i$ and $B_i$ in the completion of $\mathcal B$ with respect to $\mu$ for $i=1,\dots,k$ such that 
$$
 E=\{\{x,y\}\in [V]^2:  \exists\, i\in [k]\ \phi_i(x)=y\mbox{ or } \phi_i(y)=x\}.
$$
Also, consider the function $k'(d)$ whose definition is the same as for $k(d)$ except the maps $\phi_1,\dots,\phi_k$ are additionally required to be involutions. Cs\'oka et al~\cite[Theorem~8.3]{CsokaLippnerPikhurko16} showed that  $(d+1)/2\le k(d)\le d/2+C\sqrt d$ and  $d+1\le k'(d)\le d+C\sqrt d$ for some constant $C$ and all $d\ge 2$. (Note that trivially $k(1)=k'(1)=1$.) Using Theorem~\ref{th:main}, we can improve the upper bounds as follows, in particular determining these functions exactly, except $k(d)$ for odd~$d\ge 3$ when there are two possible values.

 \begin{theorem}\label{th:k} For every $d\ge 2$, we have $k'(d)\le d+1$ and $k(d)\le \lceil(d+2)/2\rceil$. \end{theorem}

One consequence of Theorem~\ref{th:k} is that every graphing with maximum degree $d$ admits a measurable orientation of edges with every out-degree at most $\lceil(d+2)/2\rceil$.
See Thornton~\cite{Thornton20arxiv} for a detailed study of questions of this type.
 
Unfortunately, we could not extend our proof of Theorem~\ref{th:main} to apply to non-invariant probability measures and Question~\ref{qu:Marks}, as stated, remains open for every $d\ge 4$. However, we could prove the analog of Theorem~\ref{th:main} for the following relaxation of the measurable chromatic index.
For a Borel multi-graph $\GG=(V,\mathcal B,E)$ and a probability measure $\mu$ on $(V,\mathcal B)$, the \emph{$\mu$-approximate chromatic index $\chi_{AP,\mu}'(\GG)$} is the smallest integer $k$ such that
for every $\epsilon>0$ there is a Borel set $A\subseteq V$ such that $\mu(V\setminus A)\le \epsilon$ and $\chi_{\mathcal B}'(\GG\upharpoonright A)\le k$,  where $\mathcal{G}\upharpoonright A$ denotes the Borel multi-graph on $A$ with the edge set $E\cap ([A]^2\times [\pi(\mathcal{G})])$.

For example, we have $\chi_{AP,\lambda}(\mathcal T_\alpha)=2$, where $T_\alpha$ is as in Example~\ref{ex:translation}.
(Indeed, for any $\epsilon>0$, the restriction of $\mathcal T_\alpha$ to $[0,1-\epsilon)$ has finite connectivity components and thus can be easily 2-coloured in a Borel way.) 

Bernshteyn~\cite[Theorem~1.5]{Bernshteyn19am} proved that $\chi_{AP,\mu}'(\GG)\le (1+o(1))\,\Delta(\GG)$
for every Borel graph $\GG$ and a probability measure $\mu$ on it. Here we strengthen his upper bound as follows (also extending it to multi-graphs).

\begin{theorem}\label{th:ap} Let $\GG=(V,\mathcal B,E)$ be a Borel multi-graph and let $\mu$ be a probability measure (not necessarily $E$-invariant) on $(V,\mathcal B)$. Then $\chi_{AP,\mu}'(\GG)\le \Delta(\GG)+\pi(\GG)$.
\end{theorem}

By building upon some new combinatorial ideas of this paper (namely, interated Vizing's chain introduced in Section~\ref{ItVizingChain}), Bernshteyn~\cite{Bernshteyn20arxiv2} 
presented a deterministic distributed algorithm that finds a proper $(\Delta(G)+1)$-edge-colouring of an $n$-vertex graph $G$ in polynomially many in $\Delta(G)$ and $\log n$ rounds, solving an important open problem in distributed algorithms. So, rather remarkably, a concept developed for definable graphs turned out to be useful in computer science.\medskip

The paper is organized as follows.
Section~\ref{sec:CountMult} is the main technical part of the paper and develops the combinatorial theory of augmenting chains in multi-graphs needed for our measurable results.
In Section~\ref{DefinableMultigraphs} we state several equivalent versions of the definitions made in the Introduction that are more suitable for the final proof, recall several standard constructions and show that the results from Section~\ref{sec:CountMult} can be applied in the definable context.
Finally in Section~\ref{sec:Invariant} (resp.\ Section~\ref{sec:QuasiInvariant}) we combine the previous results to prove Theorems~\ref{th:main} and \ref{th:k} (resp.\ Theorem~\ref{th:ap}.)\medskip

In terms of notation, let us point that $\mathbb N:=\{0,1,2,\dots\}$ contains 0 and that the range of integer indices like $i,j,k$ starts with $0$ unless stated otherwise. Also, recall that $[k]:=\{1,\dots,k\}$.

\section{Countable Multi-Graphs}\label{sec:CountMult}

This section is the combinatorial core of our arguments.
For better readability we decide to divide it into several subsections.
Our goal is to describe a way how to modify a given partial edge colouring of $G$ to a better one (in the sense that it contains fewer uncoloured edges).
More concretely, the construction depends on a given uncoloured edge to which we assign a sequence of edges, that we call (iterated) Vizing's chain, along which we do the improvement.
The assignment is done in such a way that every coloured edge can be a member only of constantly many (iterated) Vizing's chains (constant in $\Delta(G)+\pi(G)$ that is absolute for every partial edge colourings of $G$) while the number of edges that one uncoloured edge must ``see'' in its corresponding (iterated) Vizing's chain(s) will depend on a natural number parameter $L$ that we assign to a given colouring.
Roughly speaking, this parameter measures how difficult it is to modify the given colouring to a better one and the bigger is this number the smaller is the ratio of uncoloured/coloured edges.

In Subsection~\ref{BasicChain} we introduce the main definitions and notation that will be used throughout this section.
Subsections~\ref{AlterPath}, \ref{MultiFan}, \ref{VizingChain} and~\ref{ItVizingChain} describe how to construct the (iterated) Vizing chain for a given uncoloured edge. In Subsections~\ref{SuperbEdges} and~\ref{Estimates}, we properly define and compute the ratio between coloured and uncoloured edges that was mentioned  above.

Even though we only handle one single connected (and thus countable) multi-graph $G$ in this section, the reader may realize that the local and algorithmic nature of our constructions and definitions imply that these are in fact Borel when we work with a Borel multi-graph $\mathcal{G}$ instead of $G$.
This is made precise in Subsection~\ref{HcBorel}.

\subsection{Augmenting chains}\label{BasicChain}

Let $G=(V,E)$ be a connected multi-graph with all degrees bounded by $\Delta\in \mathbb{N}$ and edge multiplicity bounded by $\pi\in \mathbb{N}$. Thus $E$ is a subset of $[V]^2\times [\pi]$ and $\deg_G(x)\le \Delta$ for every $x\in V$, where the \emph{degree} $\deg_G(x):=|\{(\{x,y\},k)\in E\}|$ is the number of edges that contain~$x$,  counted with their multiplicities. Clearly, the vertex set $V$ is countable.

We use letters $x,y,\dots$ (resp.\  $e,f,\dots$) when we speak about vertices (resp.\ edges) of $G$.
We slightly abuse the notation and write 
\begin{itemize}
	\item $x\in f$ if one vertex of $f$ is $x$ (that is, if $f=(\{x,y\},k)$ for some $y\in V$ and $k\in[\pi]$), 
	\item $\{x,y\}=f$ if the two vertices that form $f$ are $x,y$ (that is, if $f=(\{x,y\},k)$  for some $k\in[\pi]$),
	\item $f\cap e\not=\emptyset$ if there is $x\in V$ such that $x\in f$ and $x\in e$.
	\end{itemize}

 Let $N(x)\subseteq E$ be \emph{the edge neighbourhood of $x$}, i.e., $N(x)$ consists of those $f\in E$ such that $x\in f$.

A {\it chain} is a sequence $P=(e_0,\dots )$ of edges of $G$ such that for every index $i\in \mathbb{N}$ with $e_i,e_{i+1}$ being in $P$ we have $e_i\cap e_{i+1}\not=\emptyset$, that is, every two consecutive edges in $P$ intersect. 
Let $l(P)=|P|$ denote the \emph{length} of the chain $P$, i.e., the number of edges in $P$. 
Note that a chain can be finite (possibly empty) or infinite; thus $l(P)\in \mathbb{N}\cup\{\infty\}$ and, if $P$ is finite, then $P=(e_0,\dots, e_{l(P)-1})$. The convention of labeling the first edge as $e_0$ allows us to write $P=(e_i)_{i<l(P)}$, regardless of whether $P$ is finite or not.
If $l(P)=\infty$, then  we define $l(P)-1=\infty$ in order to avoid case by case statements in several places.

We call $e_{i-1}$ the \emph{$i$-th} edge of $P$.
For an edge $f$ that occurs exactly once in $P$, let its \emph{index} $i(f)$ be $i\ge 1$ such that $f=e_{i-1}$, that is, the index of the $i$-th edge is $i$. Also, for $i\le l(P)$, let $P_i:=(e_j)_{j<i}$ denote the \emph{$i$-th prefix} of~$P$ (which consists of the first $i$ edges from $P$). We have, for example, that $P_{l(P)}=P$. 
For chains $P$ and $Q$, we write $P\sqsubseteq Q$ if $P=Q_{l(P)}$, that is, $P$ is a prefix of~$Q$.
If $P$ is a finite chain with the last edge $e$ and $Q$ is a chain with the first edge $f$ and $e\cap f\not=\emptyset$, then we write $P^\frown Q$ for the chain that is the concatenation of $P$ and~$Q$.

Let us call a chain $P=(e_i)_{i<l(P)}$ a \emph{path} if $P$ is empty, or if every vertex $z\in V$ belongs to at most 2 edges from $P$ and there is a vertex that belongs only to $e_0$. (In other words, $P$ is a finite path with a fixed direction or an infinite one-sided ray, where no self-intersections are allowed.) Also, a chain  $P$ is called a \emph{cycle} if $P$ is non-empty and every vertex belongs to 0 or 2 edges of $P$. (These are just finite cycles, having some edge and direction fixed.)

When we write $f;A\to B$ we mean that $f$ is a \emph{partial function} from $A$ to $B$, that is, a function from some subset $\dom(f)$ of $A$ to $B$. Its \emph{range} is 
$$\operatorname{rng}(f):=\{f(a): a\in \dom(f)\}.$$
 A {\it partial (edge) colouring} of $G$ is a partial function $c;E\to [\Delta+\pi]$, where $[\Delta+\pi]$ is a set of colours of size $\Delta+\pi$. Usually, we denote colours by small Greek letters, $\alpha,\beta,\dots$, etc. We assume that there is some given ordering of the colours and whenever we need to chose one of the colours we always chose the minimal possibility. A partial colouring $c$ is called \emph{proper} if every two distinct edges $e,f\in \dom(c)$ with $e\cap f\not=\emptyset$ get distinct colours, that is, $c(e)\not=c(f)$.

Let $c$ be some proper partial colouring.
We say that $c$ is {\it full} if $\dom(c)=E$, that is, every edge is coloured.
We write $U_c:=E\setminus \dom(c)$ for the set of \emph{uncoloured} edges.
Also, for $x\in V$, let 
$$
 m_c(x):=[\Delta+\pi]\setminus \{c(e):e\in N(x)\cap \dom(c)\}
 $$ 
 be the set of colours that are {\it missing} at~$x$.

\begin{claim}\label{BiggerThanMultiplicity}
We have $|m_c(x)|\ge \pi$ for every $x\in V$.
\end{claim}
\begin{proof}
There are at most $\Delta$ colours used at the vertex $x$ since $\deg_G(x)\le \Delta$.
Therefore the  number of the remaining colours must be at least~$\pi$.
\end{proof}

Next, given a proper partial colouring $c;E\to [\Delta+\pi]$, we are going to define various useful properties of a chain, each being stronger than the previous one, as follows. 

\begin{definition}\label{cAdmiss}
We say that a chain $P=(e_i)_{i<l(P)}$ is 
\begin{enumerate}[(a),nosep]
\item \emph{edge injective} if every edge appears at most once in $P$, that is, for every $0\le i<j< l(P)$ we have that $e_i\not=e_j$ as elements of $E\subseteq [V]^2\times [\pi]$,

\item \emph{$c$-shiftable} if $l(P)\ge 1$, $P$ is edge injective,  $e_0\in U_c$
and $e_j\in \dom(c)$ for every $1\le j<l(P)$ (that is, if $P$ is non-empty with no edge repeated and $e_0$ is the unique uncoloured edge of $P$);

\item \emph{$c$-proper-shiftable} if $P$ is $c$-shiftable and $c_P;E\to [\Delta+\pi]$ is a proper partial colouring, where $c_P$ is \emph{the shift of $c$ along $P$} (or \emph{$P$-shift of $c$} for short)
which is defined as 
   \begin{itemize}
   	\item $\dom(c_P)=\dom(c)\cup\{e_0\}\setminus \{e_{l(P)-1}\}$ where we put $\{e_{l(P)-1}\}=\emptyset$ if $l(P)=\infty$,
        \item $c_P(e_i)=c(e_{i+1})$ for every $i+1<l(P)$,
        \item $c_P(f)=c(f)$ for every $f\in \dom(c)\setminus P$;
    \end{itemize}
   \item \emph{$c$-augmenting} if $P$ is $c$-proper-shiftable and either $l(P)=\infty$ or $P$ is finite with $m_{c_P}(x)\cap m_{c_P}(y)\not=\emptyset$ where $x\not=y$ are the vertices of the last edge $e_{l(P)-1}$ of~$P$.
\end{enumerate}
\end{definition}

In other words, $P$ is \emph{$c$-proper-shiftable} if $P$ is non-empty, all edges in $P$ are distinct, $e_0$ is the only uncoloured edge in $P$, and if we shift the colouring $c$ down one position along $P$, then the new partial colouring $c_P$ is still proper. 
Moreover, such a chain $P$ is called \emph{$c$-augmenting} if either $P$ is infinite or $P$ is finite and its last edge $e_{l(P)-1}$ misses some colour $\beta$ at both endpoints with respect to the modified colouring $c_P$.
In the former case, the proper colouring $c_P$ colours every edge of $P$ while, in the latter case, we can achieve this
by extending $c_P$ to colour the last edge of $P$ with $\beta$.
Thus a $c$-augmenting chain $P$ gives us a way to extend a proper colouring to include a new edge $e_0$, with all modifications restricted to the edges in~$P$.
Note that the colouring $c_P$, i.e., the $P$-shift of $c$, can be defined for $c$-shiftable chains that are not necessarily $c$-proper-shiftable.
Note that a chain consisting of one uncoloured edge is always $c$-proper-shiftable 
(then $c_P$ is the same as $c$). When the partial colouring $c$ is understood, we may omit it, for example, just saying that $P$ is augmenting.

Let us state some basic properties involving the defined concepts for future reference.

\begin{claim}\label{LastNotcoloured}
Let $P=(e_j)_{j<l(P)}$ be a $c$-shiftable chain and $c_P$ be the $P$-shift of $c$.
Then $c(e_j)\not=c_P(e_j)$ for every $0<j<l(P)-1$.
Moreover, if $l(P)<\infty$ then $P\cap U_{c_P}=\{e_{l(P)-1}\}$ (that is, the last edge of $P$ is the unique edge in $P$ which is not coloured by $c_P$).\qed
\end{claim}

\begin{claim}\label{Split}
Let $P=(e_j)_{j<l(P)}$ be a $c$-shiftable chain and $i< l(P)$.
Let $Q:=(e_j)_{i\le j<l(P)}=(e_{i},\dots)$ be obtained from $P$ by removing the first $i$ edges. (Note that $P={P_{i}}^\frown Q$.)
Let $c_{i}$ be $P_{i+1}$-shift of $c$.
Then $Q$ is $c_i$-shiftable and the $P$-shift of $c$ is equal to the $Q$-shift of $c_i$.
\end{claim}
\begin{proof}
It follows from Claim~\ref{LastNotcoloured} applied to the $(i+1)$-st prefix $P_{i+1}$ (and the edge-injectivity of $P$), that $e_{i}$ is the unique edge
of $P$ not coloured by $c_i$.
Since $Q$ starts with this edge and is a subsequence of $P$, it is $c_i$-shiftable.
The claim that $P$-shift of $c$ is equal to the $Q$-shift of $c_i$ is again an easy consequence of the edge injectivity of $P$.
\end{proof}

Using the notation of the Claim~\ref{Split}, we note that the partial colouring $c_i$ need not be proper even if $c$ and the $P$-shift of $c$ are proper.
However in the sequel we use Claim~\ref{Split} only in the situations where 
$P_i$ is $c$-proper-shiftable for every $i\le l(P)$ and therefore in those cases the $P_i$-shift of $c$ is always proper.

For the purposes of this paper, we will consider only two basic $c$-proper-shiftable chains, namely, what we call a maximal alternating path and a maximal fan.
All other $c$-proper-shiftable chains that we use here will be concatenations of these two building blocks.

\subsection{Alternating paths}\label{AlterPath}

Recall that $c;E\to [\Delta+\pi]$ is a proper partial colouring.
Let $x\in V$, let $\alpha,\beta\in [\Delta+\pi]$ be different colours and suppose that $\beta\in m_c(x)$.
Then there is a unique maximal chain $P=(e_i)_{i<l(P)}$ such that $x\in e_0$ if $l(P)>0$, $x\not\in e_1$ if $l(P)>1$, and $c(e_{i})=\alpha$ (resp.\ $c(e_{i})=\beta$) for every $i<l(P)$ that is even (resp.\ odd). Informally speaking, we start with $x$ and follow the edges coloured $\alpha$ or $\beta$ as long as possible. Since  the partial colouring $c$ is proper and $\beta$ is missing at $x$,  the colours on the chain alternate between $\alpha$ and $\beta$ (starting with $\alpha$) and we never return to a vertex we have previously visited (and thus the edges in $P$ form a path). 
We call this unique maximal chain {\it the (alternating) $\alpha/\beta$-path starting at $x\in V$} and denote it as $P_c(x,\alpha/\beta)$. 
If $P_c(x,\alpha/\beta)$ is finite and non-empty, then we call the unique $y\in V$ such that $|\{f\in P_c(x,\alpha/\beta):y\in f\}|=1$ and $y\not=x$ the {\it last vertex of $P_c(x,\alpha/\beta)$}. If $P_c(x,\alpha/\beta)$ is empty (which happens exactly when $\alpha\in m_c(x)$), then the \emph{last vertex} is~$x$.
Whenever we write $P_c(x,\alpha/\beta)$ we always assume that the condition that $\beta\in m_c(x)$ is satisfied. 

The following claim summarizes some obvious properties of $\alpha/\beta$-paths.

\begin{claim}\label{Basic}
For every $x\in V$ and $\beta\in m_c(x)$, we have:
\begin{enumerate}[(i),nosep]
    \item\label{it:Bi} $P_c(x,\alpha/\beta)$ is edge injective,
    \item\label{it:Bii} $|N(z)\cap P_c(x,\alpha/\beta)|\le 2$ for every $z\in V$ and   $P_c(x,\alpha/\beta)$ is a path,
    \item\label{it:Biii} if $d;E\to [\Delta+\pi]$ is another proper partial colouring such that
    $c\upharpoonright P_c(x,\alpha/\beta)=d\upharpoonright P_c(x,\alpha/\beta)$ and $\beta\in m_d(x)$, then $P_c(x,\alpha/\beta)$ is a prefix of $P_d(x,\alpha/\beta)$.\qed
\end{enumerate}
\end{claim}

The following proposition states, in particular, that if  $e=\{x,y\}$ is an uncoloured edge with colours $\alpha\not=\beta$ missing at respectively $y$ and $x$ then we can properly shift the colouring along  the $\alpha/\beta$-path starting at $x$ down to $e$; moreover, if the path does not end in $y$ then this gives an augmenting chain.
Although all claims of the proposition are fairly routine, we include a formal proof for the sake of completeness (and similar applies to a few other results stated later).

\begin{proposition}\label{AlterPathAdmiss}
Let $e=\{x,y\}$ be an uncoloured edge.
Let $\beta\in m_c(x)$ and $\alpha\in m_c(y)$ be distinct.
Let $P:=e^\frown P_c(x,\alpha/\beta)$ be the chain obtained by prepending $e$ to the alternating $\alpha/\beta$-path starting at $x$. Then $P_i$ is $c$-proper-shiftable for every $1\le i\le l(P)$.
Moreover, if $y$ is not the last vertex of $P_c(x,\alpha/\beta)$, then $P$ is $c$-augmenting.
\end{proposition}
\begin{proof}
	By the second part of Claim~\ref{Basic}, $P=(e_i)_{i<l(P)}$ is a path or a cycle (with the latter alternative taking place if and only if $y$ is the last vertex of $P_c(x,\alpha/\beta)$).
	
	We have $e_0=e$ and therefore $e_0\in U_c$.
	The edge-injectivity of $P$ follows from the fact that $P$ is a path or a cycle.
	By definition, we have  for every $1\le i<l(P)$ that $e_i\in P_c(x,\alpha/\beta)\subseteq \dom(c)$. Thus $e_0$ is the only edge in $P$ which is not in $\dom(c)$, that is, $P$ is shiftable. Let $c_1=c$ and, for $2\le i\le l(P)$,  let $c_i$ be the shift of $c$ along $P_i$.

	
Observe that if $2\le i<\infty$, then each $c_i$ is obtained from $c$ by colouring $e_0=e$ with $\alpha$, uncolouring $e_{i-1}$, and swapping the colours $\alpha$ and $\beta$ on the intermediate edges $e_1,\dots,e_{i-2}$. Thus, when we pass from $c$ to $c_i$, the sets of used/missing colours at any vertex $v\in V$ are the same except when $v=x$ (when a new colour $\beta$ appears at $x$), $v=y$ (when the colour $\alpha$ appears at $y$) and when $v$ is in the last edge of $P_i$ (when one of the colours $\alpha$ or $\beta$ becomes missing if $v\not\in \{x,y\}$). Since $\beta\in m_c(x)$ and $\alpha\in m_c(y)$, we see that $c_i$ is a proper colouring, that is, $P_i$ is proper-shiftable whenever $i<\infty$.
This also implies that if $l(P)=\infty$, then $P_{l(P)}$ is proper-shiftable.
This is because $c_P$ can be thought of as a limit of $c_{P_i}$ for $i<l(P)=\infty$ (given any finite set of edges there is $i_0<l(P)=\infty$ such that $c_i$ agrees with $c_j$ on this set for every $i_0<i\le j\le \infty$). Since being a proper partial colouring is a local condition, the claim follows.

Finally, if $y$ is not the last vertex of $P_c(x,\alpha/\beta)$, then $P$ is a path and either $P$ is infinite or the vertices of the last edge $e$ of $P$ both miss, in the shift $c_P$, the colour in $\{\alpha,\beta\}\setminus \{c(e)\}$. In either case, the chain $P$ is augmenting.
\end{proof}

\subsection{Fan}\label{MultiFan}

As before, let $c;E\to [\Delta+\pi]$ be a proper partial colouring.
Let $e\in U_c$ and $x\in e$.
Recall that there is some fixed ordering on the set of colours $[\Delta+\pi]$.
We define the {\it maximal fan around $x$ starting at $e$}, in symbols $F_c(x,e)$, as a (finite) chain $P=(e_0,e_1,\dots,e_k)$ such that $x\in e_i$ for every $i\le k$ and if we denote the other vertex in $e_i$ by $v_i$ then the following statements are satisfied
\begin{enumerate}[(a),nosep]
    \item $e_0=e$,
    \item $P$ is edge injective,
    \item $c(e_{i+1})\in m_c(v_i)$ for every $i<k$ and $c(e_{i+1})$ is the minimal colour available in the $i$-th step,
    where we say that a colour $\alpha$ is \emph{available in the $i$-th step} if $\alpha\in m_c(v_i)$ and $\alpha\not=c(e_{j+1})$ for every $j<i$ such that $v_j=v_i$,
    \item $(e_0,\dots,e_k)$ is maximal with these properties.
\end{enumerate}
We denote the minimal colour available in the $i$-th step as $\alpha_i(c,x,e)$ or $\alpha_i$ for short if the context is understood.
This gives rise to an accompanying injective sequence of colours $(\alpha_0,\dots,\alpha_{k-1})=(c(e_1),\dots, c(e_k))$.
Note that, when $\pi\ge 2$, it is possible that $v_i=v_j$ for different indices $i,j\le k$, however it is not possible that $v_i=v_{i+1}$ for any $i+1<l(P)$.
 
In other words, we construct the fan $F_c(x,e)$ as follows. Start with $e$, denoting it as $e_0:=\{x,v_0\}$, and define $A_0:=m_c(v_0)$. (Each set $A_i$ will be exactly the set of available colours at the $i$-th step.) Suppose that we have defined $(e_0,\dots,e_{i-1})$ and non-empty sets $A_0,\dots,A_{i-1}$ for some $i\ge 1$. For $0\le j\le i-1$, let $\alpha_{j}$ be the smallest element of $A_{j}$. It will always be the case that $c(e_{j+1})=\alpha_j$ for each $0\le j\le i-2$ and we try to define $e_i$ to satisfy this condition for $j=i-1$.
If no edge at $x$ is coloured $\alpha_{i-1}$ then the current fan is maximal, we let $F_c(x,e):=(e_0,\dots,e_{i-1})$ and stop. Otherwise let $e_{i}:=\{x,v_{i}\}$ be the unique edge at $x$ coloured $\alpha_{i-1}$. If $e_{i}$ is equal to some $e_{j}$ with $j<i$, then we let $F_c(x,e)=(e_0,\dots,e_{i-1})$ and stop (without including $e_i$ into the fan). Otherwise, let $A_{i}$
be obtained from $m_c(v_{i})$  by removing $\{\alpha_j: v_j=v_{i},\ j< i\}$, that is, removing those colours that have been previously ``used'' at the current vertex~$v_{i}$.
Note that $|A_i|\ge |m_c(v_i)|-(\pi-1)>0$ so we can proceed with the next iteration step.
Since edges do not repeat, we have to stop at some point, obtaining the maximal fan $F_c(x,e)$. 

The purpose of this construction, like that in Proposition~\ref{AlterPathAdmiss}, is to generate a sequence of edges starting with a given uncoloured edge $e$ so that the shift of the current colouring $c$ along any prefix is still proper. The following claim formalizes this statement.

\begin{claim}\label{AdmissMultiFan}
Let $e\in U_c$ and $x\in e$.
Then $F_c(x,e)_i$ is $c$-proper-shiftable for every $i\le l(F_c(x,e))$.
\end{claim}
\begin{proof} Clearly, no conflict can arise at the vertex $x$ because it belongs to every edge of the fan. So we need to consider only the other endpoints $v_i$.
If some vertex $v$ appears as $v_{i_1},\dots,v_{i_k}$ in the fan, then new colours that may be introduced at $v$ during a shift are limited to $\alpha_{i_1},\dots,\alpha_{i_k}$. Since no colour is repeated in this sequence and all of them are in $m_c(v)$, no conflict can arise at $v$ either. 
\end{proof}

Like in Proposition~\ref{AlterPathAdmiss}, if a maximal fan starting with an uncoloured edge $e$ does not allow by itself to extend the domain of the current colouring to $e$, then there is some concrete obstacle for this. Here, it is the coincidence of the minimal available colours at two distinct steps, as is shown by the following proposition.

\begin{proposition}\label{TwoIndices}
Let $e\in U_c$ and $x\in e$. Let $k=l(F_c(x,e))-1$. If $F_c(x,e)$ is not $c$-augmenting, then there is $j<k$ such $\alpha_j(c,x,e)=\alpha_k(c,x,e)$.
Moreover in such a situation we must have $v_j\not=v_{k}$.
\end{proposition}
\begin{proof}
	Recall that the set of the colours available in the $k$-th step is non-empty. (Indeed, note that $|m_c(v_k)|\ge \pi$ by Claim~\ref{BiggerThanMultiplicity} while at most  $|\{i<k:v_i=v_{k}\}|<\pi$ further colours can be unavailable.)

Write $d$ for the $F_c(x,e)$-shift of $c$.
By the Claim~\ref{AdmissMultiFan} we have that $d$ is a proper partial colouring and by Claim~\ref{LastNotcoloured} we have $e_k\not\in \dom(d)$.
It is easy to see from the definition of $d$ and $\alpha_k$ that $\alpha_k\in m_d(e_k)$.
Since $F_c(x,e)$ is not $c$-augmenting, we have that $\alpha_k$ does not belong to $m_d(x)= m_c(x)$, that is, some edge at $x$ has $c$-colour $\alpha_k$. The only reason why this edge is not added as $e_{k+1}$ to the maximal fan $F_c(x,e)$ is that it already appears in the fan, that is, there is some $j< k$ with $\alpha_k=c(e_{j+1})=\alpha_j$, proving the first conclusion of the proposition.

Also, the case $v_j=v_k$ is impossible because the colour $\alpha_j=\alpha_k$ that was `used' at the $j$-th step is unavailable at every later moment when we visit the same vertex $v_j$ again. 

Thus the obtained index $j$ has all the required properties.
\end{proof}

\subsection{The Vizing Chain}\label{VizingChain}

As usual, let $c;E\to [\Delta+\pi]$ be a proper partial colouring. Here, for every $e\in U_c$ and $x\in e$, we define a chain $V_c(x,e)$, which we call the {\it Vizing chain}, that starts at $e$ and is always $c$-augmenting. (In particular, as we discuss at the end of this section, this suffices for establishing Vizing's theorem for finite multi-graphs.)

If the fan $F_c(x,e)$ is $c$-augmenting, then we define 
 \begin{equation}\label{eq:Wcxe0}
 V_c(x,e):=F_c(x,e).
 \end{equation}
 
Now assume that  the Vizing fan $F_c(x,e)$ is not $c$-augmenting (and let this assumption apply until the end of Section~\ref{SuperbEdges}).
Informally speaking, we consider the two special indices $j\not=k$ with the same available colour $\beta$ (whose existence is guaranteed by Proposition~\ref{TwoIndices}), fix $\alpha\in m_c(x)$, and consider two alternating $\alpha/\beta$-paths starting with $v_j$ and $v_k$. For at least one choice of $i\in\{j,k\}$, the $\alpha/\beta$-path starting at $v_i$ does not end in~$x$. Thus, if we shift colours in the fan $F_c(x,e)$ until $e_i$ is uncoloured and then shift colours down to $e_i$ along the whole alternating $\alpha/\beta$-path at $v_i$, then we extend the domain of the colouring to $e$, apart at most one edge. This exceptional edge (if exists) is the last edge of the path and can be properly coloured with one of $\alpha$ or $\beta$. 
The following proposition establishes the above claims and formally defines Vizing's chain $V_c(x,e)$ in this case.  

\begin{proposition}\label{VizingPath}
Let $e\in U_c$ and $x\in e$ be such that $F_c(x,e)$ is not $c$-augmenting. 
Let $F_c(x,e)=(e_0,\dots,e_k)$ and let $\alpha\in m_c(x)$ be the minimal colour in $m_c(x)$.
Then there is $i\in \{j,k\}$ (where $j$ is the index from Proposition~\ref{TwoIndices})
such that if we write $\beta:=\alpha_j(c,x,e)=\alpha_k(c,x,e)$, then the chain
    \begin{equation}\label{eq:Wcxe}
    V_c(x,e):={F_c(x,e)_{i+1}}^\frown P_c(v_i,\alpha/\beta)
    \end{equation}
     is $c$-augmenting and, moreover, the path $P_c(v_i,\alpha/\beta)$ does not use the vertex~$x$.
\end{proposition}

\begin{proof}
If $F_c(x,e)$ is not $c$-augmenting, then by Proposition~\ref{TwoIndices} we have an index $j$ such that $v_j\not=v_{k}$ and $\beta=\alpha_j=\alpha_k$.
In particular, $\beta\in m_c(v_j)\cap m_c(v_k)$.
  
The $\alpha/\beta$-alternating paths $P_c(v_j,\alpha/\beta)$ and $P_c(v_k,\alpha/\beta)$ cannot both use the vertex~$x$. Otherwise, since $\alpha\in m_c(x)$, the union of these two paths will be  a connected graph with all degrees 2 apart three distinct vertices (namely $x,v_i,v_k$) of degree 1 each, which is clearly impossible.
Thus we can pick $i\in\{j,k\}$ be such that no edge of $P_c(v_i,\alpha/\beta)$ contains~$x$; if both $j$ and $k$ satisfy this, we let $i:=j$. This in particular satisfies the second claim of the proposition, namely that $P_c(v_i,\alpha/\beta)$ does not use $x$.

Denote as $c'$ the $F_c(x,e)_{i+1}$-shift of $c$.
By the proof of Claim~\ref{AdmissMultiFan}, we have that $e_i\in U_{c'}$ and $\beta\in m_{c'}(v_i)$.
Let us show that 
 \begin{equation}\label{eq:Pc=Pc'}
  P_c(v_i,\alpha/\beta)=P_{c'}(v_i,\alpha/\beta).
  \end{equation}
Here we have to distinguish two cases.

Suppose first that $i=j$. When we pass from $c$ to $c'$, we modify colours only on the edges $e_0,\dots,e_j$, all of which are incident ot $x$. None of the changed colours can be $\alpha$ (because $\alpha\in m_c(x)$) or $\beta$ (because $e_{j+1}$, the unique edge at $x$ of colour $\beta$, keeps its colour). 
Now,~\eqref{eq:Pc=Pc'} trivially follows. 

Suppose now that $i=k$. First, let us show that $P_c(v_i,\alpha/\beta)$ uses neither $v_j$ nor $v_{j+1}$.
Recall that $v_k\not=v_j$ by
Proposition~\ref{AlterPathAdmiss}.
If $P_c(v_k,\alpha/\beta)$ uses $v_j$, then $v_j$ is the last vertex of the path because $\beta\in m_c(v_j)$; however then $P_c(v_j,\alpha/\beta)$, as the reversed $P_c(v_k,\alpha/\beta)$, does not use $x$, contradicting the choice of~$i$.
Suppose next that $P_c(v_k,\alpha/\beta)$ uses $v_{j+1}$. Then the path also uses $x$ since the edge $e_{j+1}=\{x,v_{j+1}\}$ has colour $\beta$ under~$c$, again contradicting the choice of $i$. Note that when we pass from $c$ to $c'$, no re-colouring involves the colour $\alpha$ for the same reason as in the case $i=j$. Also, the colour $\beta$ is shifted only once, from $e_{j+1}$ to $e_{j}$. Since the path $P_c(v_k,\alpha/\beta)$ does not use any vertex of $e_{j+1}\cup e_{j}=\{x,v_i,v_{j+1}\}$, it equals $P_{c'}(v_k,\alpha/\beta)$. This finishes the proof of~\eqref{eq:Pc=Pc'}.


Now it follows from Proposition~\ref{AlterPathAdmiss} that $P_{c'}(v_i,\alpha/\beta)$, which is equal to $P_{c}(v_i,\alpha/\beta)$ and thus avoids $x$, is  $c'$-augmenting.
By Claim~\ref{Split}, the $P_{c'}(v_i,\alpha/\beta)$-shift of $c'$ is the same as the $V_c(x,e)$-shift of $c$. By combining all this with Claim~\ref{AdmissMultiFan}, we see that $V_c(x,e)$ is
$c$-augmenting.
\end{proof}

We call the index $i$ in Proposition~\ref{VizingPath} the {\it first critical index} and let
 \begin{equation}\label{eq:PcxeDef}
 P_c(x,e):=P_c(v_i,\alpha/\beta)
 \end{equation}
 (Note that  the colours $\alpha$, $\beta$ and the index $i$ are uniquely determined by $x,e$ here.)
With this notation, the corresponding Vizing chain from~\eqref{eq:Wcxe} is
$$	
V_c(x,e)={F_c(x,e)_{i+1}}^\frown P_c(x,e).
$$
See Figure~\ref{fi:VC} for an illustration.
\begin{figure}
\includegraphics{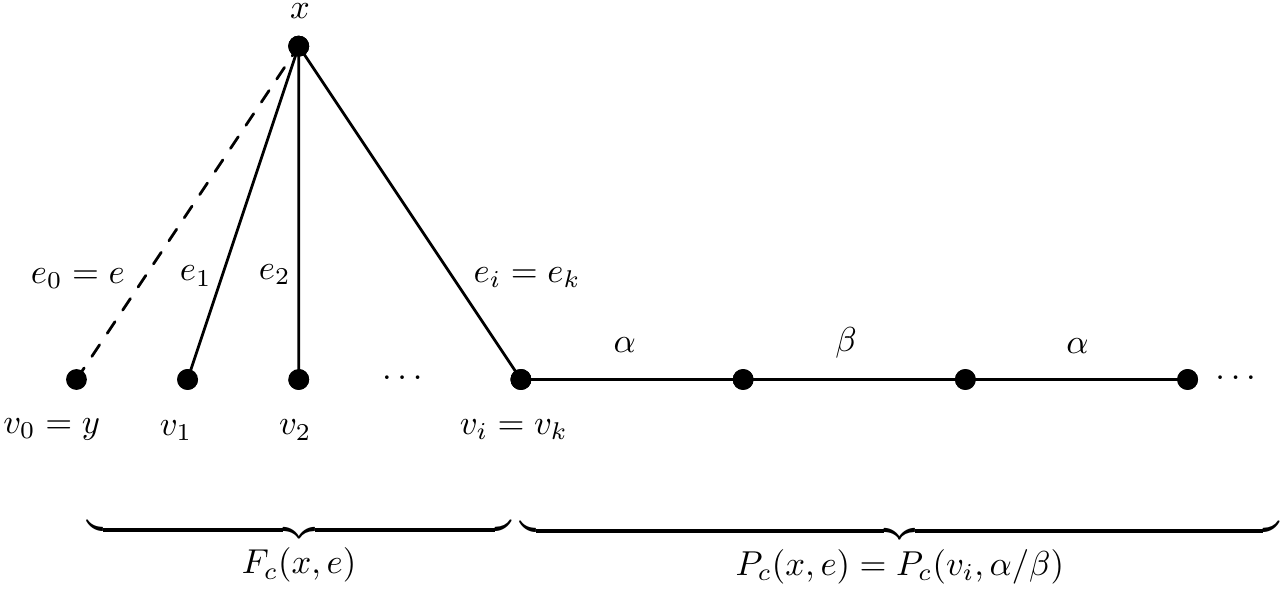}
\caption{Vizing's chain $V_c(e,x)$ with the first critical index $i=k$}
\label{fi:VC}
\end{figure}

Now, Vizing's Theorem for finite multi-graphs can be easily derived.
Suppose that $c;E\to [\Delta+\pi]$ is proper partial colouring that is maximal in the sense that we cannot find another proper partial colouring $d$ with the property that $\dom(c)\subsetneq \dom(d)$.
Suppose that there is some $e\in U_c$ and consider the Vizing chain $V:=V_c(x,e)$ for some $x\in e$.
Since the chain $V$ is edge injective, it is finite.
Thus when we pass to $c'$, the $V$-shift of $c$, the last edge  $e_{l(V)-1}$ of the chain
becomes uncoloured. Since $V$ is $c$-augmenting by Proposition~\ref{VizingPath}, there is a colour missing in $c'$ at both end-points of $e_{l(V)-1}$ and we can extend the proper colouring $c'$ to this edge. 
Then $\dom(c)\subsetneq \dom(c')$, which is a contradiction.

Note that this argument, as stated, does not work for countably infinite multi-graphs since we cannot assume the existence of such a maximal colouring $c$ and if we want to build the colouring by induction it is not clear that we end up with a full colouring, i.e., some of the edges may change their colour infinitely often.
This is caused by the fact that the lengths of the arising chains need not be uniformly bounded (or can be even infinite).
However since the condition in the definition of a proper colouring is local, an easy compactness argument shows that there exists some full proper colouring $c:E\to [\Delta+\pi]$ even in the countably infinite case.

\subsection{Iterated Vizing's chain}\label{ItVizingChain}

The last type of $c$-proper-shiftable chains that we need is the iterated version of the Vizing chain.
 Its definition requires some work and consists of three cases, appearing in~\eqref{eq:IVC1}, \eqref{eq:IVC2} and~\eqref{eq:IVC3}. Such chains are not needed for the proof of Theorem~\ref{th:ap}, the approximate version of Vizing's theorem. The reader interested only in this theorem may skip all forthcoming definitions and results where iterated chains occur.

First, let us very informally describe how we construct an iterated Vizing's chain.
As usual, we have $e\in U_c$ and $x\in e$. Recall that we assume that the fan $F_c(x,e)$ is not augmenting.
We shift the fan until the first critical index~$i$. Now, if we were to follow the above proof of Vizing's theorem, we would be augmenting the current colouring using the alternating path $P_c(x,e)$. Instead, we pick an edge $f$ on this path, shift the colouring along the path so that the selected edge $f$ becomes uncoloured (calling this colouring $c'$), and then construct the augmenting Vizing chain for $c'$ starting with $y$, the farthest end-point of $f$. The 
corresponding  iterated Vizing's chain will be the concatenation of all involved edges (namely, the fan at $x$ until the index~$i$, the first alternating path until $f$, and finally the augmenting Vizing chain for $f\in U_{c'}$ and $y\in f$).
However, in order  to avoid the issues when the second Vizing chain uses an edge on which $c$ and $c'$ differ, we define the fan around $y$ slightly differently in fact. Also, we find it more convenient to define the fan at $y$ without using the shift $c'$, i.e., in terms of $c$ only. 

Let us now give the proper definition as well as detailed explanations.

\begin{definition}\label{de:suitable}
We say that $f\in P_c(x,e)$ is \emph{suitable} if
\begin{enumerate}[(a),nosep]
    \item\label{it:suitable1} the graph distance of $f$ and $e$ is more than $3$, i.e., $l(P)>5$ for every path $P$ such that $e,f\in P$,
    \item\label{it:suitable2} $f$ is not the last edge of the chain $P_c(x,e)$,
    \item\label{it:suitable3} $c(f)$ is the minimal colour missing at $x$ (in our notation it is always $\alpha$).
\end{enumerate}
\end{definition}

Recall that every second edge of $P_c(x,e)$ has colour $\alpha$ under the colouring $c$ and, for such edges, the notion of being suitable is just a mild technical restriction. Trivially,
only a constant number (roughly, at most 
$2\Delta^4$) of colour-$\alpha$ edges on $P_c(x,e)$ are not suitable.
The main purpose of this definition is to make sure that the edges $e$ and $f$ are far apart and so the re-colouring of the fan around $x\in e$ does not affect the colours around $f$.
Some of the required consequences of this definition are stated in the following claim.

\begin{claim}\label{Suitable}
Let $P_c(x,e)$ be the $\alpha/\beta$-path corresponding to $e\in U_c$ and $x\in e$. Let $f\in P_c(x,e)$ be suitable.
Denote as $c_f$ the $V_c(x,e)_{i(f)}$-shift of $c$. Let $y$ be the last vertex of $P_c(x,e)_{i(f)}$ and let $z\in f\setminus \{y\}$, i.e., $z$ is the other vertex of $f$.
Then 
\begin{enumerate}[(i), nosep]
    \item\label{it:Suitable1} $m_{c_f}(y)=m_c(y)\cup\{\alpha\}$,
    \item\label{it:Suitable2} $m_{c_f}(z)=m_c(z)\cup \{\beta\}$,
    \item\label{it:Suitable3} $m_{c_f}(u)=m_c(u)$ for every $u\in g$ where $g\in N(y)\setminus f$ and $u\not=y,z$,
    \item\label{it:Suitable4} $\alpha\not\in m_{c_f}(z)$,
    \item\label{it:Suitable5} $\beta\not\in m_{c_f}(y)$.
\end{enumerate}
\end{claim}
\begin{proof}
Let $i$ be the first critical index of $F_c(x,e)$ and $c_i$ be the $F_c(x,e)_{i+1}$-shift of $c$.
Denote the last edge of $F_c(x,e)_{i+1}$ as $e_i$ and write $v_i$ for the other vertex in $e_i$ (other than $x$).
It follows from the Claim~\ref{Split} that $c_f$ is the $P$-shift of $c_i$ where $P:={e_i}^\frown P_c(x,e)_{i(f)}$.
Moreover we have from the proof of Proposition~\ref{VizingPath} (specifically from~\eqref{eq:Pc=Pc'}) that $P_c(x,e)=P_{c_i}(v_i,\alpha/\beta)$.

The assumption that $f$ is suitable (namely, Property~\ref{it:suitable1} from the definition) immediately implies that $m_{c}(r)=m_{c_i}(r)$ for every $r\in g$ where $g\in N(y)$ (i.e., for every neighbour vertex of $y$ including  $z$)
as well as for $r=y$.
Also we have $c(f)=c_i(f)=\alpha$ and $u\not=v_i$ for every $u$ satisfying \ref{it:Suitable3} of our claim.
Since $P$ is a path and the only colours that are modified when we pass from $c_i$ to $c_f$ are $\alpha$ and $\beta$,
we see that $m_{c_f}(r)=m_{c_i}(r)$ for every $r\in V\setminus \{x,y,z,v_i\}$.
This proves~\ref{it:Suitable3}.
Items \ref{it:Suitable1}, \ref{it:Suitable2} and~\ref{it:Suitable4} follow from the fact that $f$ is the last edge of $P$, i.e., the colour $\alpha$ at $y$ (resp.\ $\beta$ at $z$) is shifted away from this vertex.
For \ref{it:Suitable5} we only need to recall that Property~\ref{it:suitable2} from the definition of a suitable edge states that $f$ is not the last edge of $P_{c}(x,e)$.
\end{proof}

Suppose that the alternating colours in $P_c(x,e)$ are $\alpha$ and $\beta$ (starting with $\alpha\in m_c(x)$).
Let $f\in P_c(x,e)$ be suitable and $y\in V$ be the last vertex of $P_c(x,e)_{i(f)}$.
We define {\it the maximal $\alpha/\beta$-conditional fan starting at $f$}, denoted as $F_c(x,e\leadsto f)$, as a chain $P=(g_0,\dots, g_m)$ such that $y\in g_i$ for every $i\le m$ and, if we denote the other vertex of $g_i$ by $u_i$, then the following is satisfied
\begin{enumerate}[(a),nosep]
    \item\label{it:cond1} $g_0=f$,
    \item\label{it:cond2} $P$ is edge injective,
    \item\label{it:cond3} $c(g_{i+1})\in m_c(u_{i})$ and it is the minimal available colour (where a colour $\delta$ is \emph{available in the $i$-th step} if $\delta\not=c(e_{j+1})$ for every $j<i$ such that $u_j=u_i$),
    \item \label{it:cond4} $\alpha,\beta \not \in m_c(u_i)$ for every $i<m$,
    \item\label{it:cond5}  if $\alpha,\beta\not\in m_c(u_m)$, then $(g_0,\dots, g_m)$ is maximal with the properties above.
\end{enumerate}
Note that we should rather write $u_i^f$, $g_i^f$ and $y^f$ to stress that those objects depend on the choice of $f$.
This will be however omitted in the cases when we work with only one $f$.

Intuitively we define $F_c(x,e\leadsto f)$ as follows.
Consider $c'$, the $V_c(x,e)_{i(f)}$-shift of $c$.
Then use the same construction as in the case of the original fan with parameters $y$ and $f$  but with the lists $m_c(u)$ (not $m_{c'}(u)$, the point is that we want  $\beta$ not to be in the list that corresponds to $z\in f$, $z\not=y$).
The construction terminates either from the same reasons as in the original fan construction or, and this is important, if we reach a vertex $u_m$ such that $\alpha$ or $\beta$ is in $m_c(u_m)$ (note that in such a case $u_m\not=z$ since $z$ is an internal vertex of an $\alpha/\beta$-alternating path).

\begin{proposition}\label{ConditionalFanAdmiss}
Let $f\in P_c(x,e)$ be suitable.
Then $P:={V_c(x,e)_{i(f)-1}}^\frown F_c(x,e\leadsto f)$ is $c$-proper-shiftable.
\end{proposition}
\begin{proof}
Clearly, $P$ is shiftable so we only need to show that the partial colouring $c'$, the $P$-shift of $c$, is proper.
It follows from Proposition~\ref{VizingPath} that the shift $c_f$ of $c$ along $V_c(x,e)_{i(f)}$ is a proper partial colouring.
Let $y\in f$ be the last vertex of $P_c(x,e)_{i(f)}$.
Suppose that we keep the same order on the colours except we make the colour $\beta$ be the biggest in this order, and with this new ordering we define $Q:=F_{c_f}(y,f)$.
If we show that $F_c(x,e\leadsto f)\sqsubseteq Q$, then this finishes the proof because it follows from Claim~\ref{AdmissMultiFan} that the prefix $Q_i$ is $c_f$-proper-shiftable for every $i\le l(Q)$.  

Recall that the elements of $F_c(x,e\leadsto f)$ are denoted as $(g_0,\dots,g_m)$ and $g_i=\{y,u_i\}$ for $i\le m$.

By Claim~\ref{Suitable} we have for every $j\le m=l(F_c(x,e\leadsto f))-1$  that $m_c(u_j)=m_{c_f}(u_j)$ if $u_j\not=u_0$ and that $m_c(u_j)\cup\{\beta\}=m_{c_f}(u_j)$ if $u_j=u_0$.
Note that the edges where $c$ and $c_f$ differ are restricted to those containing $x$
and to the edges of the $\alpha/\beta$-alternating path until the vertex $y$. Thus, since $y$ is suitable, we have
 $c(g)=c_f(g)$ for every $g\in N(y)\setminus \{f\}$.

Suppose on the contrary that there is $j<m$ such that $F_c(x,e\leadsto f)_{j+1}\sqsubseteq Q$ but $F_c(x,e\leadsto f)_{j+2}\not\sqsubseteq Q$.
Thus $g_{j+1}$ is the first edge of $F_c(x,e\leadsto f)$ which either is not present or is in a different position in $Q$.

Suppose first that $u_j\not=u_0$. We have $m_{c_f}(u_j)=m_c(u_j)$. Also,  $\alpha,\beta\not\in m_c(u_j)$ by Property~\ref{it:cond4} of the definition of $F_c(x,e\leadsto f)$.
Therefore the minimal available colour in the $j$-th step is in both cases the same, denote it as $\delta$.
Thus $c(g_{j+1})=\delta$. Since $c(g_{j+1})=c_f(g_{j+1})$, we must have that the $(j+2)$-nd edge of $Q$ is
also $g_{j+1}$, a contradiction.

Suppose now that $u_j=u_0$. Consider the step when we construct the $(j+2)$-nd edge of $Q$, after having
constructed $(g_0,\dots,g_j)\sqsubseteq Q$.
Here we have $|m_{c_f}(u_j)|\ge \pi +1$.
Since $|\{k<j:u_k=u_0\}|<\pi$ we can find at least two available colours.
It follows from the re-ordering of the colours that the minimal colour, call it $\delta$, is not~$\beta$. Also, $\delta\in m_{c_f}(u_j)$ cannot be equal to $\alpha$ which is present at $u_j=u_0=z$ in~$c_f$ by Claim~\ref{Suitable}\ref{it:Suitable4}. The same $\delta$ is the minimal available colour at the $j$-th step for $F_c(x,e\leadsto f)$ because $m_c(z)=m_{c_f}(z)\setminus\{\beta\}$ by Claim~\ref{Suitable}\ref{it:Suitable2}. Thus, we have again that $c(g_{j+1})=\delta$ and, since $c_f(g_{j+1})=c(g_{j+1})$, that  the edge $g_{j+1}$ is also included as
the $(j+2)$-nd edge into $Q$. This contradicts the choice of~$j$.
\end{proof}

We distinguish three types of suitable edges.
We say that a suitable $f\in P_c(x,e)$ is
\begin{itemize}
    \item of {\it Type $0$} if ${V_c(x,e)_{i(f)-1}}^\frown F_c(x,e\leadsto f)$ is augmenting,
    \item of {\it Type I} if it is not of Type $0$ and $\beta\in m_c(u_m)$ (recall that $g_m=\{y,u_m\}$ is the last edge of $F_c(x,e\leadsto f)$),
    \item of {\it Type II} if it is not of Type $0$ or Type I.
\end{itemize}

Let us first make the following easy observation.

\begin{claim}\label{cl:AlphaT0}
If $f$ is a suitable edge and $\alpha\in m_c(u_m)$, then $f$ is of Type $0$.
\end{claim}
\begin{proof}
Let $c'$ be the ${V_c(x,e)_{i(f)-1}}^\frown F_c(x,e\leadsto f)$-shift of $c$.
It follows from the Claim~\ref{Split} that $c'$ is also the $F_c(x,e\leadsto f)$-shift of $c_f$.
By Claim~\ref{Suitable}\ref{it:Suitable1}, where we have $f=\{y,z\}$, we have $m_{c_f}(y)=m_c(y)\cup\{\alpha\}$. By Items~\ref{it:Suitable2}
and~\ref{it:Suitable4} of Claim~\ref{Suitable}, we have $\alpha\not \in m_c(z)$ and thus $u_m$ cannot be equal to $u_0=z$.
Claim~\ref{Suitable}\ref{it:Suitable3} gives that $m_{c_f}(u_m)=m_c(u_m)$; in particular, $\alpha\in m_{c_f}(u_m)$.
Note that during the shift from $c_f$ to $c'$ no colours $\alpha$ and $\beta$ are changed.
This implies that $\alpha\in m_{c'}(u_m)\cap m_{c'}(y)$, as required.
\end{proof}

Our aim is now to define the {\it iterated Vizing chain} $W_c(x,e\leadsto f)$ for every $e\in U_c$, $x\in e$ and some suitable $f\in P_c(x,e)$.
We handle each type separately.

\medskip\noindent{\bf Type 0.}
If $f\in P_c(x,e)$ is a suitable edge of Type $0$ then we put
\begin{equation}\label{eq:IVC1}
 W_c(x,e\leadsto f):={V_c(x,e)_{i(f)-1}}^\frown F_c(x,e\leadsto f).
 \end{equation}
 \comm{
Note that if $\alpha\in m_c(u_m)$, then $f$ is of Type $0$.
To see this denote as $c'$ the $W_c(x,e\leadsto f)$-shift of $c$.
It follows from the Claim~\ref{Split} that $c'$ is also the $F_c(x,e\leadsto f)$-shift of $c_f$.
By Claim~\ref{Suitable} we have $m_{c_f}(y)=m_c(y)\cup\{\alpha\}$ and $u_m\not=u_0$ because $\alpha\not \in m_c(u_0)$.
Another use of Claim~\ref{Suitable} yields $m_{c_f}(u_m)=m_c(u_m)$, especially $\alpha\in m_{c_f}(u_m)$.
It remains to realize that during the shift from $c_f$ to $c'$ no colours $\alpha$ and $\beta$ are changed.
This implies that $\alpha\in m_{c'}(u_m)\cap m_{c'}(y)$.
}

\medskip\noindent{\bf Type I.}
Suppose that a suitable edge $f$ is of Type I.
By the definition we have $\beta\in m_c(u_m)$ and $|F_c(x,e\leadsto f)|=m+1$. Also,
$\alpha\not\in m_c(u_m)$ by Claim~\ref{cl:AlphaT0}.
We call $m$ the {\it second critical index}.
Note that $u_0\not=u_m$ because, since $f$ is suitable, we must have $\beta\not\in m_c(u_0)$ (as $|P_c(x,e)|\ge 3$
by Definition~\ref{de:suitable}\ref{it:suitable1}).

\comm{Our aim is to define \comment{OP: we should be precise when $W_c(x,e\leadsto f)$ is / is not defined}
\begin{equation}\label{eq:IVC2}
 W_c(x,e\leadsto f)={V_c(x,e)_{i(f)-1}}^\frown {F_c(x,e\leadsto f)}^\frown P_{c}(u_m,\alpha/\beta).
 \end{equation}
In order to do that we need to rule out some edges.
}

\begin{definition}\label{de:SuperbI}
We say that a suitable $f\in V_c(x,e)$ of Type I is \emph{superb of Type I} if, in the above notation,
$P_c(u_m,\alpha/\beta)=P_{c_f}(u_m,\alpha/\beta)$. In this case we define 
$$P_c(x,e\leadsto f):=P_c(u_m,\alpha/\beta)
$$
 and
\begin{equation}\label{eq:IVC2}
 W_c(x,e\leadsto f):={V_c(x,e)_{i(f)-1}}^\frown {F_c(x,e\leadsto f)}^\frown P_{c}(u_m,\alpha/\beta).
 \end{equation}
\end{definition}

\begin{figure}
\includegraphics{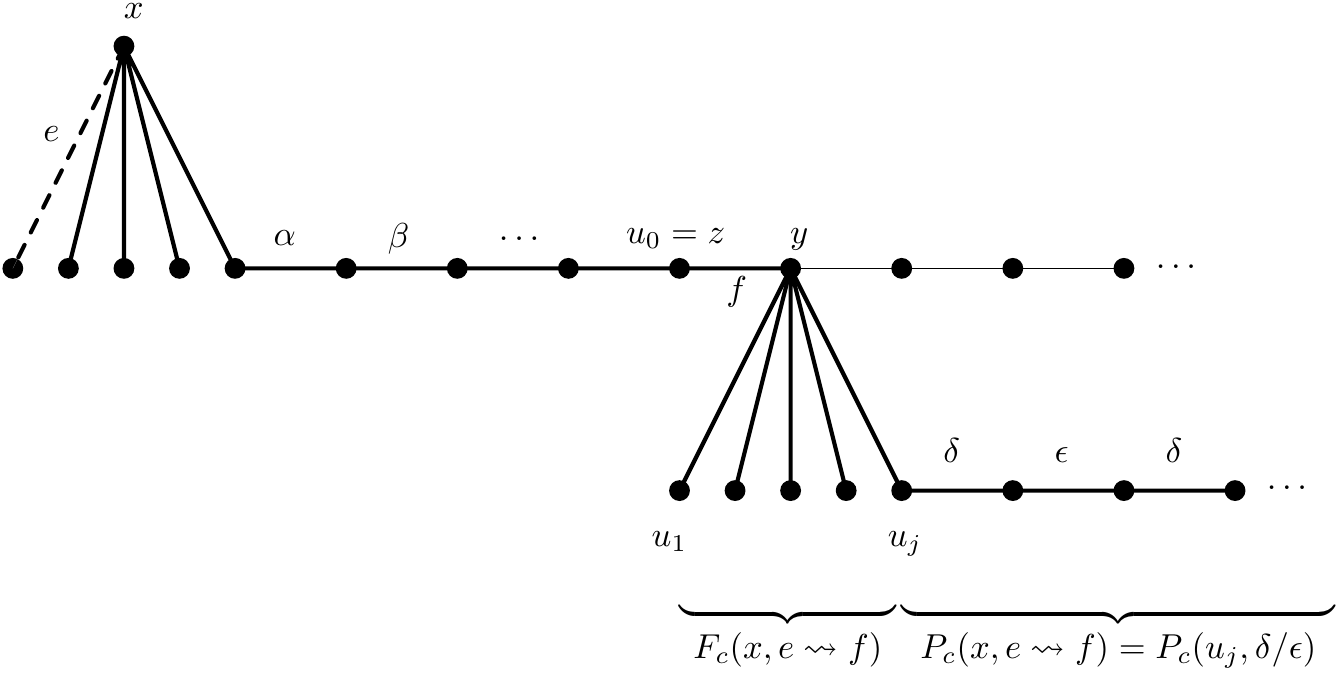}
\caption{Iterated Vizing's chain $W_c(x,e\leadsto f)$ (edges in bold) for a suitable edge $f$ of Type I or II 
}
\label{fi:IVC}
\end{figure}

For an illustration, see Figure~\ref{fi:IVC} (where, for a Type I edge $f$, we have $\epsilon=\alpha$ and $\delta=\beta$).
Note that $W_c(x,e\leadsto f)$ is undefined if $f$ is a suitable but not superb edge of Type~I. 
The requirement that $f$ is superb is needed for the following result.

\begin{proposition}\label{TypeI}
Let $f\in V_c(x,e)$ be superb of Type I.
Then $W_c(x,e\leadsto f)$ is $c$-augmenting.
\end{proposition}
\begin{proof}
First we show that $W_{c}(x,e\leadsto f)$ is edge injective.
By Proposition~\ref{ConditionalFanAdmiss} we have that
$$P:={V_c(x,e)_{i(f)-1}}^\frown {F_c(x,e\leadsto f)}$$
is edge injective and moreover $c'$, the $P$-shift of $c$, is also the $F_c(x,e\leadsto f)$-shift of $c_f$.
This implies that $W_c(x,e\leadsto f)$ is not edge injective if and only if there is $h\in P_c(x,e\leadsto f)\cap P$.
The critical observation is that since $f$ is superb of Type I we must have $c(h)=c_f(h)\in \{\alpha,\beta\}$.
However the shift from $c_f$ to $c'$ does not use edges of colour $\alpha$ and $\beta$, in particular $c'(h)=c_f(h)$.
But since $h\in P$ we must have $c'(h)\not=c(h)$ by Claim~\ref{LastNotcoloured}. Therefore $W_c(x,e\leadsto f)$ is edge injective.

Similar argument shows that $P_c(u_m,\alpha/\beta)=P_{c_f}(u_m,\alpha/\beta)=P_{c'}(u_m,\alpha/\beta)$.
Namely, the first equality holds since $f$ is superb of Type I and the second equality follows from the fact that no edges with colour $\alpha$ or $\beta$ were modified in the shift from $c_f$ to $c'$.
Then Proposition~\ref{AlterPathAdmiss} together with Claim~\ref{Split} imply that $W_c(x,e\leadsto f)$ is $c$-proper-shiftable.

By Claim~\ref{Split} we have that proving that $W_c(x,e\leadsto f)$ is $c$-augmenting is the same as proving that ${g_m}^\frown P_c(x,e\leadsto f)={g_m}^\frown P_{c'}(u_m,\alpha/\beta)$ is $c'$-augmenting.
To show this we use Proposition~\ref{AlterPathAdmiss} once again, for which we need to check
that $y\in f$ is not the last vertex of ${g_m}^\frown P_c(x,e\leadsto f)$.
This is easy since $P_c(x,e\leadsto f)=P_c(u_m,\alpha/\beta)$ and we know that $y$ is not the last vertex of any $\alpha/\beta$-path (with respect to $c$) that goes through $y$ since $f$ is suitable.
\end{proof}

\medskip\noindent{\bf Type II.}
Suppose that a suitable $f\in P_c(x,e)$ is of Type II.
Let $\delta$ be the smallest colour in $m_c(y)$. The reason why we cannot extend $F_c(x,e\leadsto f)$ is the same as in Proposition~\ref{TwoIndices}, namely that there is a colour $\epsilon$ and an index $i<m=l(F_c(x,e\leadsto f))-1$ such that $\epsilon$ is the minimal colour available in both $m_c(u_i)$ and $m_c(u_m)$.
It is clear that $\delta\not=\epsilon$ because $f$ is not of Type $0$ and $\{\alpha,\beta\}\cap \{\delta,\epsilon\}=\emptyset$ because $f$ is not of Type I.

Consider now the alternating $\delta/\epsilon$-paths
$P_c(u_i,\delta/\epsilon)$ and
$P_c(u_m,\delta/\epsilon)$.
Our aim is to choose one of them, call it $Q$, and then define
$$W_c(x,e\leadsto f):={V_c(x,e)_{i(f)-1}}^\frown {F_c(x,e\leadsto f)_{j+1}}^\frown Q,$$
where $j\in \{i,m\}$, depending on the choice of $Q$, is such that $W_c(x,e\leadsto f)$ is $c$-augmenting.
As in the case of Type I, we need to rule out some edges.

\begin{definition}
We say that a suitable $f\in V_c(x,e)$ of Type II is \emph{superb of Type II} if, in the above notation, both of the following equalities hold:
\begin{itemize}
    \item $P_c(u_i,\delta/\epsilon)=P_{c_f}(u_i,\delta/\epsilon)$,
    \item $P_c(u_m,\delta/\epsilon)=P_{c_f}(u_m,\delta/\epsilon)$.
\end{itemize}
\end{definition}

Let $f\in V_c(x,e)$ be superb of Type II.
Then we put
\begin{equation}\label{eq:IVC3}
 W_c(x,e\leadsto f):={V_c(x,e)_{i(f)-1}}^\frown {F_c(x,e\leadsto f)_{j+1}}^\frown P_c(u_{j},\delta/\epsilon),
 \end{equation}
where $j\in \{i,m\}$ is the index that satisfies the analogue of Proposition~\ref{VizingPath} with respect to $c_f$ and $F_c(x,e\leadsto f)$, i.e., there is no $h\in P_c(u_j,\delta/\epsilon)$ such that $y\in h$.
If both indices $i$ and $m$ satisfy this, then we put $j:=i$.
We call this index $j$ the {\it second critical index} and write 
$$
 P_c(x,e\leadsto f):=P_c(u_{j},\delta/\epsilon).$$
See Figure~\ref{fi:IVC} for an illustration.

\begin{proposition}
Let $f\in V_c(x,e)$ be superb of Type II.
Then $W_c(x,e\leadsto f)$ is $c$-augmenting.
\end{proposition}
\begin{proof}
First we show that $W_c(x,e\leadsto f)$ is edge injective.
By Proposition~\ref{ConditionalFanAdmiss} we have that
$$P:={V_c(x,e)_{i(f)-1}}^\frown {F_c(x,e\leadsto f)}_{j+1}$$
is edge injective and moreover $c'$, the $P$-shift of $c$, is also the $F_c(x,e\leadsto f)_{j+1}$-shift of $c_f$.
This implies that $W_c(x,e\leadsto f)$ is not edge injective if and only if there is $h\in P_c(x,e\leadsto f)\cap P$.
Suppose that such an $h$ exists.
We know that $c(h)=c_f(h)\in \{\epsilon,\delta\}$.
This implies that $y\in h$ or $x\in h$, i.e., it is not the case that $h\in P_c(x,e)$.
We know that $y\not \in h$ by the definition of the second critical index and $W_c(x,e\leadsto f)$.
But then we have $c_f(h)=c'(h)$ because $f$ is suitable, i.e., far away from $x$.
However since $h\in P$ we must have $c'(h)\not =c(h)$.
This is a contradiction because by the assumption we have $c(h)=c_f(h)=c'(h)$.

A similar argument shows that $P_c(u_j,\delta/\epsilon)=P_{c_f}(u_j,\delta/\epsilon)=P_{c'}(u_j,\delta/\epsilon)$.
Here the first equality is the assumption and the second follows from the choice of the second critical index (namely from~\eqref{eq:Pc=Pc'}). 
Then Proposition~\ref{AlterPathAdmiss} together with Claim~\ref{Split} imply that $W_c(x,e\leadsto f)$ is $c$-proper-shiftable.

By the Claim~\ref{Split} we have that proving that $W_c(x,e\leadsto f)$ is $c$-augmenting is the same as proving that ${g_j}^\frown P_c(x,e\leadsto f)={g_j}^\frown P_{c'}(u_j,\delta/\epsilon)$ is $c'$-augmenting.
This follows easily from Proposition~\ref{AlterPathAdmiss} because from the definition of the second critical index we know that for every $h\in P_c(x,e\leadsto f)$ we have $y\not\in h$.
This finishes the proof.
\end{proof}

\subsection{How many superb edges are there?}\label{SuperbEdges}

We conclude the definition of Vizing's and iterated Vizing's chain by an estimate on superb edges.
First we artificially extend the notion of being superb to Type $0$ in order to unify the presentation.

\begin{definition}
We call a suitable edge $f\in P_c(x,e)$ \emph{superb} if it is of Type $0$ (in which case we set
$P_c(x,e\leadsto f):=\emptyset$) or superb of Type I or superb of Type II.
\end{definition}

The main results of the previous subsection then directly imply the following proposition.

\begin{proposition}
Let $f\in P_c(x,e)$ be superb.
Then $W_c(x,e\leadsto f)$ is $c$-augmenting.\qed
\end{proposition}

The following result, where we do not try to optimise the constants, gives that there are many superb edges. 

\begin{proposition}\label{SuperbEdgesSamecolour}
For every $L\le l(P_c(x,e))$ there are colours $\gamma,\theta\in [\Delta+\pi]$ such that
$$|H_{\gamma,\theta}|\ge \frac{1}{3(\Delta+\pi)^2}\left(\frac{L}{2}-\Delta^5-1\right)-2\Delta^3,$$
 where  $H_{\gamma,\theta}=H_{\theta,\gamma}$ consists of those superb edges $f\in P_c(x,e)_L$ such that the alternating path $P_c(x,e\leadsto f)$ is coloured by $\{\gamma,\theta\}$ (or its subset).
\end{proposition}


\begin{proof}
There are at most $(\frac{L}{2}+\Delta^5+1)$ edges that are not suitable.
In other words, there are at least $(\frac{L}{2}-\Delta^5-1)$ suitable edges.
It follows that at least $\frac{1}{3}$ of those edges is either of Type $0$ or of Type I or of Type II.
We examine each situation separately.

{\bf Type $\bf{0}$.}
In this case every such edge is superb, the choice of colours is irrelevant and for any $\gamma,\theta$ we have 
$$|H_{\gamma,\theta}|\ge \frac{1}{3}\left(\frac{L}{2}-\Delta^5-1\right),$$
under the assumption that at least $\frac{1}{3}$ of the edges are of Type $0$.

{\bf Type I.} Let us show that we can take $\{\alpha,\beta\}$ for $\{\gamma,\theta\}$ where, as usual, we denote the colours on $P_c(x,e)$ by $\alpha$ and $\beta$.

\begin{claim}\label{cl:FSuitNotSuperb} If a suitable edge $f\in V_c(x,e)$ is not superb then it is at distance 1 from the last vertex of $P_c(x,e)$ or $P_c(x,\beta/\alpha)$.\end{claim}

\begin{proof}[Proof of Claim] Given $f$, we use the same notation as above. In particular, $F_x(x,e\leadsto f)=(g_0,\dots,g_m)$ with $g_t=\{y,u_t\}$. We will show that, in fact, 
$u_m$ is the last vertex of $P_c(x,e)$ or $P_c(x,\beta/\alpha)$. 

By the definition we know that if a suitable $f\in V_c(x,e)$ is not superb then $P_c(u_m,\alpha/\beta)\not=P_{c_f}(u_m,\alpha/\beta)$. Let $Q$ be the maximal common prefix of these two paths.
We know that we must have $u_m\not=u_0$, $\alpha\not\in m_c(u_m)=m_{c_f}(u_m)$ and $\beta\in m_c(u_m)=m_{c_f}(u_m)$. 
This implies that $P_c(u_m,\alpha/\beta)$ and $P_{c_f}(u_m,\alpha/\beta)$ start with the same edge (and colour $\alpha$) and thus $l(Q)\ge 1$ and $c\upharpoonright Q=c_f\upharpoonright Q$.
Write $w$ for the last vertex of $Q$.
There are two cases that we treat separately.

First suppose that $l(P_c(u_m,\alpha/\beta))>l(Q)$.
Denote the first new edge in $P_c(u_m,\alpha/\beta)$ as~$h$.
We must have $c(h)\in \{\alpha,\beta\}$ and $c_f(h)\not\in \{\alpha,\beta\}$ because $c_f$ is a proper partial colouring and $l(Q)\ge 1$.
Consider now the shift from $c$ to $c_f$.
The only possible edges that were coloured with $c$ by $\alpha$ or $\beta$ and changed the colour to something different from $\alpha$ and $\beta$ were $f$ and $e_{j+1}$ where the index $j$ is from Proposition~\ref{TwoIndices}.
Thus $h$ is $f$ or $e_{j+1}$. 

First we handle the case $h=f$. Recall that $f=\{y,u_0\}$ and $w$ is the last vertex of~$Q$. It is impossible that $w=u_0$ because the only edge at $N(u_0)$ with $c_f$-colour in $\{\alpha,\beta\}$ is $g:=P_c(x,e)_{i(f)-2}$, the edge before $f$ on $P_c(x,e)$; however, $c_f(g)=\alpha$ is different from $c(g)=\beta$, contradicting $g\in Q$ and $c\upharpoonright Q=c_f\upharpoonright Q$. Thus the last vertex of $Q$ is $w=y$. Since $\alpha$ is missing at
$y$ under $c_f$, we have that in fact $Q=P_{c_f}(u_m,\alpha/\beta)$. It follows that $P_{c_f}(u_m,\alpha/\beta)$ is the reversal of $P_{c_f}(y,\beta/\alpha)$ and consequently $u_m$ is the last vertex of $P_c(x,e)$, as desired.
 
Next, let us derive a contradiction by assuming that $h=e_{j+1}$. Since $\alpha\in m_c(x)$, it holds that $x$ (resp.\ $e_{j+1}$) is the last vertex (resp.\ edge) of $P_c(u_m,\alpha/\beta)$.
The first critical index of $F_c(x,e)$ cannot be $j$
(otherwise 
$c_f(e_{j+1})=c(e_{j+1})=\beta$, contradicting the maximality of~$Q$).
This means that $P_c(v_j,\alpha/\beta)$ contains $x$.
This can only happen when $x$ (resp.\ $e_{j+1}$) is the last vertex (resp.\ edge) of $P_c(v_j,\alpha/\beta)$.
The same holds for $P_c(u_m,\alpha/\beta)$, which implies that $u_m=v_{j}$.
But this is impossible since $f$ is suitable, i.e, far away from $e$.

The second case is when $l(P_{c_f}(u_m,\alpha/\beta))>l(Q)=l(P_c(u_m,\alpha/\beta))$.
Denote the first new edge in $P_{c_f}(u_m,\alpha/\beta)$ as $h$.
We have $c_f(h)\in \{\alpha,\beta\}$ but $c(h)\not \in \{\alpha,\beta\}$.
Consider again the shift from $c$ to $c_f$.
The only possible edges that were not coloured with $c$ by $\alpha$ or $\beta$ but  got such a colour after the shift are $e_j$ and $e_k$, where we again use the notation from Proposition~\ref{TwoIndices}.
Therefore $w$, the last vertex of $Q$, is equal to one of $v_j,x,v_k$.
First let us derive a contradiction by assuming that $w=v_t$ with
$t\in\{j,k\}$. By the assumption that $Q=P_c(u_m,\alpha/\beta)$ we have that $P_c(v_t,\alpha/\beta)$ is just reversed $P_c(u_m,\alpha/\beta)$.
If $t=j$, then this would imply that $j$ is the first critical index because $P_c(v_t,\alpha/\beta)$ being an inverse path to $P_c(u_m,\alpha/\beta)$ does not contain $x$ (because if $x$ were contained, it would be the last vertex but $x\not=u_m$ since $f$ is suitable).
Write $g$ for the first edge in $P_c(v_j,\alpha/\beta)$, i.e., $c(g)=\alpha$ and $v_j\in g$.
We also have that $g$ is the last edge of $P_c(u_m,\alpha/\beta)=Q$.
Because $f$ is suitable, the length of $P_c(v_j,\alpha/\beta)$ is at least 2 and we have that $c_f(g)=\beta$. Since $Q\sqsubseteq P_{c_f}(u_m,\alpha/\beta)$ and they both start with colour $\alpha$ we must have $\alpha=c(g)=c_f(g)=\beta$, which is a contradiction.
Suppose now that $t=k$.
Then $k$ must be the first critical index because otherwise $P_{c_f}(u_m,\alpha/\beta)$ has $v_k$ as the last vertex.
(This follows from Proposition~\ref{TwoIndices} since $v_j\not=v_k$ and if $j$ is the first critical index, then $e_k$ does not change the colour when we shift $c$ to $c_f$.)
But then the same argument with $g$, the first edge of $P_c(v_k,\alpha/\beta)$, as above shows that this is impossible.
Thus $w=x$, that is, $x$ is the last vertex of $Q=P_c(u_m,\alpha/\beta)$.
We have that $P_c(x,\beta/\alpha)$ is just the reversed $P_c(u_m,\alpha/\beta)$ and $u_m$ is the last vertex of $P_c(x,\beta/\alpha)$, finishing the proof of the claim.\end{proof}

By Claim~\ref{cl:FSuitNotSuperb} there are trivially at most $2\Delta^2$ choices of $f\in V_c(e,x)$ which is  suitable but not superb of Type~I.
We conclude that 
$$|H_{\alpha,\beta}|\ge \frac{1}{3}\left(\frac{L}{2}-\Delta^5-1\right)-2\Delta^2,$$
under the assumption that at least $\frac{1}{3}$ of suitable vertices are of Type I.

{\bf Type II.}
There are at most $(\Delta+\pi)^2$ choices for the colours $\gamma,\theta$.
Counting argument shows that there is a choice of colours $\gamma,\theta$ such that the size of the set of those suitable edges of Type II that uses them is at least $\frac{1}{3(\Delta+\pi)^2}(\frac{L}{2}-\Delta^5-1)$.

Since $\{\alpha,\beta\}\cap \{\gamma,\theta\}=\emptyset$, the only reason why $f$ is suitable and not superb is that, up to swapping $\gamma$ and $\theta$, one of $P_{c}(u_i,\gamma/\theta)$ or $P_c(u_m,\gamma/\theta)$ visits some neighbour of $x$.
For each such neighbour $w$ we consider a maximal $\gamma/\theta$-path $Q_w$ that goes through $w$. This path is unique up to its direction.
There are two directions that we can follow along $Q_w$ and therefore there are two possible endpoints $r,s$ on $Q_w$ that can serve as $u_i^f$ or $u_m^f$.
This gives that there are  at most $2\Delta$ such vertices.
Clearly, there are at most $\Delta^2$ edges for each of these vertices that can play the role of $f$ that is suitable of Type II but not superb of Type II.
This gives that 
$$|H_{\gamma,\theta}|\ge \frac{1}{3(\Delta+\pi)^2}\left(\frac{L}{2}-\Delta^5-1\right)-2\Delta^3,$$
under the assumption that at least $\frac{1}{3}$ of suitable edges is of Type II.
\end{proof}

\subsection{Double counting}\label{Estimates}

We use double counting argument in two settings, one for Vizing's chains and one for iterated Vizing's chains.

\subsubsection{Vizing's chains}
Let $c;E\to[\Delta+\pi]$ be a proper partial colouring.

\begin{definition}\label{SimpleImprovement}
We say that $c$ \emph{cannot be improved in $L$ steps} if, for every $e\in U_c$ and $x\in e$,
the fan $F_c(x,e)$ is not $c$-augmenting and 
$$|P_c(x,e)|\ge L.$$

\end{definition}

We will use a double counting argument inside the following bipartite graph (where it is convenient to view edges as ordered pairs).

\begin{definition}
We write $\mathcal{H}_c$ for the bipartite graph with parts $U_c$ and $\dom(c)$ where
$$(e,f)\in E(\mathcal{H}_c) \ \Leftrightarrow \ (\exists x\in e) \ f\in V_c(x,e).$$
\end{definition}

The crucial observation is that the degree of every $f\in \dom(c)$ in $\mathcal{H}_c$ is bounded by a constant (that is, by a function that depends only on $\Delta+\pi$). Here, as in the rest of this section, we do not optimise constants as their values are irrelevant in our applications of the stated bounds.

\begin{proposition}\label{EstimateSimpleVizing}
Let $f\in \dom(c)$.
Then $\deg_{\mathcal{H}_c}(f)\le (\Delta+\pi)^4$.
\end{proposition}
\begin{proof}
Let $f\in \dom(c)$.
If $f\in V_c(x,e)$ for some $e\in U_c$ and $x\in e$, then either $f\in F_c(x,e)$ or $f\in P_c(x,e)$.

Suppose first that ${f\in F_c(x,e)}$.
We must have $x\in f$.
It is immediate that there are at most $2\Delta-1$ such choices of $e\in U_c$.

It remains to consider ${f\in P_c(x,e)}$.
In this case there is a colour $\delta\not=c(f)$ such that $\{\delta,c(f)\}$ are the colours in the alternating path $P_c(x,e)$.
By the definition $x$ must be a neighbour of one of the endpoints of $P_c(x,e)$ while $e$ is an edge containing~$x$.
This implies that there are at most $2(\Delta+\pi-1)\Delta^2$ choices of $e\in U_c$.

Putting all together we have
$$\deg_{\mathcal{H}_c}(f)\le 2\Delta-1+2(\Delta+\pi-1)\Delta^2\le (\Delta+\pi)^4$$
and that finishes the proof.
\end{proof}

On the other hand, if a colouring $c$ cannot be improved in a small number of steps, then every vertex in the other part $U_c$ of the graph $\mathcal{H}_c$ has large degree:

\begin{proposition}\label{EstimateSimpleVizingDown}
Suppose that $c$ cannot be improved in $L$ steps.
Then $\deg_{\mathcal{H}_c}(e)\ge L$ for every $e\in U_c$.
\end{proposition}
\begin{proof}
This follows immediately from the assumption since $|P_c(x,e)|\ge L$ for every $e\in U_c$ and $x\in e$.
\end{proof}

The last two results together may be intuitively interpreted for finite graphs as follows: if a partial colouring cannot be improved in $L$ steps, then the ratio of uncoloured to coloured edges within Vizing's chains is at most $(\Delta+\pi)^4/L$ (from which it can be deduced that at most $O(1/L)$-fraction of all edges can be uncoloured).

\subsubsection{Iterated Vizing's chains}

\begin{definition}
We say that $c$ cannot be \emph{iteratively improved in $L$ steps} if it cannot be improved in $L$ steps and
\begin{equation}
\label{eq:IILS}
|P_c(x,e\leadsto f)|\ge L
 \end{equation}
for every $e\in U_c$ and $x\in e$, and for every edge $f\in P_c(x,e)$ that is superb and satisfies $|P_c(x,e)_{i(f)}|\le L$.
\end{definition}

Note that if $L>0$ above, then there are no superb edges of Type $0$ among the first $L$ edges of any $P_c(x,e)$.

\begin{definition}
We write $\mathcal{H}^\leadsto_c$ for the bipartite graph with parts $U_c$ and $\dom(c)$ where
$$(e,g)\in E(\mathcal{H}_c^\leadsto) \ \Leftrightarrow \ (\exists x\in e) \ (\exists f\in P_c(x,e) \ \operatorname{superb}) \ \ g\in W_c(x,e\leadsto f).$$
\end{definition}

Surprisingly, even in $\mathcal{H}_c^\leadsto$, the degree of every $g\in \dom(c)$ can be bounded by a function that depends only on $\Delta+\pi$.

\begin{proposition}\label{EstimateVizing}
Let $g\in \dom(c)$.
Then $\deg_{\mathcal{H}^\leadsto_c}(g)\le (\Delta+\pi)^9$.
\end{proposition}
\begin{proof}
There are three possible positions for $g$, namely $g\in V_c(x,e)$, $g\in F_c(x,e\leadsto f)$ or $g\in P_c(x,e\leadsto f)$ for some $e\in U_c$, $x\in e$ and $f\in P_c(x,e)$ that is superb.

The case ${g\in V_c(x,e)}$ is examined in Proposition~\ref{EstimateSimpleVizing}, by which we have at most $(\Delta+\pi)^4$ many such $e\in U_c$.

If ${g\in F_c(x,e\leadsto f)}$, then there are at most  $2\Delta$ possibilities for such an $f$ because $f\cap g\not=\emptyset$. Another use of Proposition~\ref{EstimateSimpleVizing} then gives at most $2\Delta(\Delta+\pi)^4$ possible choices of $e\in U_c$.

Finally, suppose that ${g\in P_c(x,e\leadsto f)}$.
As in Proposition~\ref{EstimateSimpleVizing}, there must be a colour $\delta\not=c(g)$ such that $\{\delta,c(g)\}$ are colours in the alternating path $P_c(x,e\leadsto f)$.
By the definition there must be a vertex $y\in V$ that is a neighbour of one of the endpoints of $P_c(x,e\leadsto f)$ and such that $y\in f$.
This gives an estimate on possible number of such edges $f$, namely there are at most $2(\Delta+\pi-1)\Delta^2$ of them.
Using Proposition~\ref{EstimateSimpleVizing} again we have at most $(\Delta+\pi)^4$ possible edges $e\in U_c$ for each such $f$.
This implies that there are at most $2(\Delta+\pi-1)\Delta^2(\Delta+\pi)^4$ choices of $e\in U_c$ such that there are $x\in e$ and superb $f\in V_c(x,e)$ with $g\in P_c(x,e\leadsto f)$.

Putting all together we obtain that
$$\deg_{\mathcal{H}_c^\leadsto}(g)\le (\Delta+\pi)^4+2\Delta(\Delta+\pi)^4+2(\Delta+\pi-1)\Delta^2(\Delta+\pi)^4\le (\Delta+\pi)^9.$$
\end{proof}

If $c$ cannot be iteratively improved in $L$ steps, then we can strengthen the conclusion of Proposition~\ref{EstimateSimpleVizingDown}
by getting a quadratic in $L$ lower bound on the degrees in $U_c$ (and this will be crucial for the proof of Theorem~\ref{th:main})

\begin{proposition}\label{EstimateVizingDown}
Suppose that $c$ cannot be iteratively improved in $L$ steps.
Then
$$\deg_{\mathcal{H}_c^\leadsto}(e)\ge \frac{L}{2\Delta}\left(\frac{1}{3(\Delta+\pi)^2}\left(\frac{L}{2}-\Delta^5-1\right)-2\Delta^3\right)$$
for every $e\in U_c$.
\end{proposition}
\begin{proof}
Let $e\in U_c$. Fix any $x\in e$.
By our assumption on $c$ and $L$, we have that $|P_c(x,e)|\ge L$.
Using Proposition~\ref{SuperbEdgesSamecolour} for this $L$, we find colours $\delta$ and $\epsilon$ such that
$$|H_{\delta,\epsilon}|
\ge \frac{1}{3(\Delta+\pi)^2}\left(\frac{L}{2}-\Delta^5-1\right)-2\Delta^3.$$

Let  $f\in H_{\delta,\epsilon}$ be arbitrary. By the definition of $H_{\delta,\epsilon}$, we have $|P_c(x,e)_{i(f)}|\le L$ and $f$ is superb. By our assumption (that is, by~\eqref{eq:IILS}),
the path $P_c(x,e\leadsto f)$ has at least $L$ edges, each of which gives a neighbour of $e$ in~$\mathcal{H}_c^\leadsto$.

We have to be careful as some edges may be overcounted this way. Let us show that for any edge $g$ there is at most $2\Delta$ choices of $h\in  H_{\delta,\epsilon}$ with $P_c(x,e\leadsto h)$ containing $g$. Assume that $c(g)\in\{\delta,\epsilon\}$ as otherwise no such $h$ can exist. Let $P$ be the maximal $\delta/\epsilon$-path containing $g$ (unique up its direction).
Every path $P_c(x,e\leadsto h)$ that contains $g$ as above has to be equal to $P$ or the reversal of~$P$. Thus all possible $h$ are confined to edges at distance 1 from  one of the two endpoints of $P$. Also, since $h\in P_c(x,e)$ has to be suitable and thus must have colour $\alpha=\min m_c(x)$, there are at most $2\Delta$ choices of $h$.

\hide{
Also,  there are at most $2\Delta$ edges $h\in H_{\delta,\epsilon}$ such that $P_c(x,e\leadsto f)=P_c(x,e\leadsto h)$ or $P_c(x,e\leadsto f)$ is a reversed $P_c(x,e\leadsto h)$.
This is because a fixed path $P_c(x,e\leadsto f)$ can have at most $2$ endpoints and every neighbour of these endpoints can be element of at most one such $h$ (this is because $c$ is proper and $h$ is suitable, therefore $c(h)=c(h')$ for every $h,h'\in H_{\delta,\epsilon}$).
}

This implies that 
$|H_{\delta,\epsilon}|\cdot L\le \deg_{\mathcal{H}_c^\leadsto}(e)\cdot 2\Delta$,
finishing the proof by our lower bound on~$|H_{\delta,\epsilon}|$.
\end{proof}

Again, an intuitive interpretation is that under the assumption that $c$ cannot be iteratively improved in $L$ steps, the fraction of uncoloured edges inside iterated Vizing's chains (and thus overall) is $O(1/L^2)$.

\section{Definable Multi-Graphs}\label{DefinableMultigraphs}

In this section we unify all the definitions concerning Borel multi-graphs and definable chromatic numbers.
For the sake of completeness we include several standard constructions.
The reader that is familiar with the basic concepts of measurable graph combinatorics should only briefly check our notations and skip to the end of this section where we show in the Proposition~\ref{HisBorel} that the constructions from the Section~\ref{sec:CountMult} are in fact Borel.

Recall from the Introduction that a {\it Borel multi-graph} with multiplicity bounded by $\pi\in \mathbb{N}$ is a triple $\mathcal{G}=(V,\mathcal{B},E)$ where $(V,\mathcal{B})$ is a standard Borel space and $E\subseteq [V]^2\times [\pi]$ is a  Borel subset.
(A set $A\subseteq [V]^2$ is Borel if and only if the corresponding symmetric set 
$\bigcup_{\{x,y\}\in A} \{(x,y),(y,x)\}$ is a Borel subset of $V\times V$.
Note that the set $[V]^2$ endowed with this Borel structure is a standard Borel space.)
We say that $\mathcal{G}$ is of bounded maximum degree if there is $\Delta\in \mathbb{N}$ such that $\deg_\mathcal{G}(x)=|\{e\in E:x\in e\}|\le \Delta$.
In that case we write $\Delta(\GG),\pi(\GG)\in \mathbb{N}$ to be the maximal degree and multiplicity of $\mathcal{G}$.
We denote the equivalence relation on $V$ that is generated by the connected components of $\mathcal{G}$ as $F_\mathcal{G}$.
Note that $F_\mathcal{G}$ is a countable Borel equivalence relation (see Section~\ref{se:CBER} below).
We write $(E,\mathcal{C})$  for the standard Borel space of edges of $\mathcal{G}$, i.e., $\mathcal{C}$ is the restriction of the standard Borel structure from $[V]^2\times [\pi(\mathcal{G})]$ to $E$.
There are Borel maps $s_0,s_1:E\to V$ that assign to each edge its vertices.
(For example, fix a Borel total order on $V$ and let $s_0(e)$ and $s_1(e)$ be respectively the smaller
and the larger vertices of~$e$.)
Note that since $\mathcal{G}$ is a multi-graph, we may have distinct $e,f\in E$ with $\{s_0(e),s_1(e)\}=\{s_0(f),s_1(f)\}$.
We write $\mathcal{E}:=(E,\mathcal{C},I_\mathcal{G})$ for the the {\it intersection graph} (or {\it line graph}) on $E$, i.e., $\{e,f\}\in I_{\mathcal{G}}$ if and only if $e\not=f$ and there are $i,j\in \{0,1\}$ such that $s_i(e)=s_j(f)$.
Then $I_\mathcal{G}\subseteq [E]^2$ and $\mathcal{E}$ is a Borel graph with all degrees bounded by $2\Delta(\mathcal{G})-2$.
We write $F_\mathcal{E}$ for the countable Borel equivalence relation on $E$ that is generated by the connected components of $\mathcal{E}$. 
Usually, we do not mention the corresponding Borel $\sigma$-algebras, i.e., we write $\mathcal{G}=(V,E)$, etc.

A {\it Borel chromatic number} $\chi_\mathcal{B}(\mathcal{G})$ of a Borel multi-graph $\mathcal{G}$ is the minimal $k\in \mathbb{N}$ such that there is a full proper Borel vertex colouring $d:V\to [k]$ of $\mathcal{G}$.
Similarly we define the {\it Borel chromatic index} $\chi'_\mathcal{B}(\mathcal{G})$ as the minimal $k\in \mathbb{N}$ such that there is a full proper Borel (vertex) colouring $c:E\to [k]$ of $I_\mathcal{G}$.
It is easy to see that with our notation
$$\chi'_\mathcal{B}(\mathcal{G})=\chi_\mathcal{C}(\mathcal{E}).$$
Compare both definitions with the ones given in the Introduction.
Note also that the subscript in the definition refers to the concrete corresponding Borel $\sigma$-algebra.
However in the sequel we write simply $B$ to refer to the Borel $\sigma$-algebra, i.e., $\chi_B(\mathcal{G})$ or $\chi_B(\mathcal{E})$.
It will be always clear from the context to which Borel $\sigma$-algebra we are referring. 

\subsection{Countable Borel equivalence relations}\label{se:CBER} 

For the convenience of the reader we recall basic notions and properties of countable Borel equivalence relations that can be found for example in \cite{KechrisMiller:toe}.
A \emph{countable Borel equivalence relation} is a pair $(X,F)$ where $X$ is a standard Borel space (we do not mention the Borel $\sigma$-algebra) and $F\subseteq X\times X$ is a Borel equivalence relation with cardinality of each equivalence class at most countable, for example $(V,F_\mathcal{G})$ or $(E,F_{\mathcal{E}})$.
For a given $A\subseteq X$ we write $[A]_F$ for the {\it $F$-saturation of $A$}, i.e., $y\in [A]_F$ if there is $x\in A$ such that $(x,y)\in F$.
If $A\subseteq X$ is Borel, then so is $[A]_F$.
We write $\mathcal{P}(X)$ for the space of all Borel probability measures on $X$.
We say that $\mu\in \mathcal{P}(X)$ is {\it $F$-invariant} (or $\mathcal{G}$-invariant in the case when $F=F_\mathcal{G}$) if $\mu(A)=\mu(B)$ for every Borel subsets $A,B\subseteq X$ such that there is a Borel bijection $\phi:A\to B$ satisfying $(x,\phi(x))\in F$ for every $x\in A$.
It is a standard fact that in the situation where $\mathcal{G}=(V,E)$ and $\mu\in \mathcal{P}(V)$ the definition of $E$-invariant measure from the Introduction coincide with $\mathcal{G}$-invariant ($F_\mathcal{G}$-invariant) measure given here.
We denote the set of all $F$-invariant measures as $\mathcal{I}_F$.
We say that $\mu\in \mathcal{P}(X)$ is {\it $F$-quasi-invariant} if for every Borel set $A\subseteq X$ we have $\mu([A]_F)=0$ if and only if $\mu(A)=0$.
We denote the set of all $F$-quasi-invariant measures as $\mathcal{QI}_{F}$.

\subsection{Measures on $E$}

Let $\mathcal{G}=(V,E)$ be a Borel multi-graph with bounded maximum degree.
Let $\mu\in \mathcal{P}(V)$.
We describe a way how to produce $\hat\mu\in \mathcal{P}(E)$ that reflects some properties of $\mu$.
Let $A\subseteq E$ be Borel, i.e., the corresponding symmetrization $\{(x,y,k)\in V\times V\times [\pi(\mathcal{G})]: (\{x,y\},k)\in A\}$ is Borel in $V\times V\times [\pi(\mathcal{G})]$.
For $x\in V$,  we define 
$$e_A(x):=|\{(\{x,y\},k)\in A\}|$$
to be the number of edges in $A$ (counting their multiplicities) that contain~$x$.
By the assumption we have $e_A(x)\le \Delta(\mathcal{G})$ for every $x\in V$ and it follows from the Lusin-Novikov Uniformisation Theorem (see e.g.,~\cite[Theorem~18.10]{Kechris:cdst}) that $e_A:V\to \{0,\dots,\Delta(\mathcal{G})\}$ is a Borel function.
Note also that $e_E(x)=\deg_{\mathcal{G}}(x)$ for every $x\in V$.
Finally we put
$$\hat\mu(A):=\frac{1}{\int_V \deg_\mathcal{G}(x)d\mu}\int_V e_A(x)d\mu.$$
It is a standard fact that $\hat\mu$ is a Borel probability measure on $E$. 
Compare this definition with the measures $M_{L}$ and $M_{R}$ defined on a countable Borel equivalence relation $F$ for example in \cite[Section~8]{KechrisMiller:toe}.

\begin{proposition}\label{MeasureHat}
Let $\mu \in \mathcal{P}(V)$.
Then $\hat\mu\in \mathcal{P}(E)$ satisfies the following:
\begin{enumerate}[(i), nosep]
    \item if $\hat\mu(A)<\epsilon$, then $\mu(\{x\in V:(\exists e\in A) \ x\in e\})<\Delta(\mathcal{G})\epsilon$,
    \item if $\mu$ is $\mathcal{G}$-invariant, then $\hat\mu$ is $\mathcal{E}$-invariant,
    \item if $\mu$ is $\mathcal{G}$-quasi-invariant, then $\hat\mu$ is $\mathcal{E}$-quasi-invariant.
\end{enumerate}
\end{proposition}
\begin{proof}
By the definition and simple computation we have $$\hat\mu(A)=\frac{1}{\int_V \deg_\mathcal{G}(x)d\mu}\int_V e_A(x)d\mu\ge \frac{1}{\Delta(\mathcal{G})}\,\mu(\{x\in V:(\exists e\in A) \ x\in e\}).$$
This proves the first item.

To prove the second item fix some Borel injection $i:A\to B$ where $A,B\subseteq E$ are Borel.
We may suppose that each $A$ and $B$ are independent in the intersection graph $\mathcal{E}$.
This follows from the fact that the degree of $\mathcal{E}$ is bounded and therefore by Theorem~\ref{th:KST} we have $\chi_B(\mathcal{E})<\infty$.
Define $A_0=s_0(A)$ and $B_0=s_0(B)$.
We have
$\mu(A_0)=\mu(s_1(A))$ by $\mathcal{G}$-invariance of $\mu$ and therefore
$$\hat\mu(A)=\frac{1}{\int_V \deg_\mathcal{G}(x)d\mu}\int_V e_A(x)d\mu=\frac{2\mu(A_0)}{\int_V \deg_\mathcal{G}(x)d\mu}.$$
Similarly we have
$$\hat\mu(B)=\frac{2\mu(B_0)}{\int_V \deg_\mathcal{G}(x)d\mu}.$$
The map $j:A_0\to B_0$ that sends $x$ to $s_0(i(e))$ where $e\in A$ is the unique edge in $A$ such that $x=s_0(e)$ is clearly a Borel injection.
It follows that $\mu(A_0)=\mu(B_0)$ and we are done.

The third item easily follows from the first part.
Namely, suppose that $\hat\mu(A)=0$.
Then $\mu(\{x\in V:(\exists e\in A) \ x\in e\})=0$ and therefore $\mu([\{x\in V:(\exists e\in A) \  x\in e\}]_{F_\mathcal{G}})=0$ by the assumption.
This implies that $\hat{\mu}([A]_{F_{\mathcal{E}}})=0$ by the definition of $\hat\mu$.
\end{proof}

\subsection{Quasi-invariant measures}

Let $(X,F)$ be a countable Borel equivalence relation and $\mu\in \mathcal{P}(X)$.
We describe the standard construction of $\tilde\mu\in \mathcal{Q}\mathcal{I}_F$ such that $\mu\ll \tilde{\mu}$ (that is, $\mu$ is absolutely continuous with respect to $\tilde{\mu}$) and $\mu([A]_F)=\tilde\mu([A]_F)$ for every Borel set $A\subseteq X$.
First we use the Feldman-Moore theorem (see e.g.,\ \cite[Theorem~1.3]{KechrisMiller:toe}) to find a sequence $\{i_n\}_{n<\omega}$ of Borel involutions of $X$ that \emph{graph} $F$, i.e., for every $(x,y)\in F$ there is $n<\omega$ such that $i_n(x)=y$.
We always assume that $i_0=\mathrm{id}_X$ is the identity map on~$X$.
We define $\tilde{\mu}$ as
$$\tilde{\mu}(A)=\sum_{n<\omega}\frac{i_n^*\mu(A)}{2^{n+1}}$$
where $A\subseteq X$ is a Borel set and $i_n^*\mu(A)=\mu(i_n^{-1}(A))$ (that is, $i_n^*\mu$ is the push-forward of $\mu$ along~$i_n$).

\begin{proposition}\label{BasicQuasi}
Let $(X,F)$ be a countable Borel equivalence relation and $\mu\in \mathcal{P}(X)$.
Then
\begin{enumerate}[(i), nosep]
    \item\label{it:BQ1} $\tilde{\mu}\in \mathcal{P}(X)$,
    \item\label{it:BQ2} $\tilde\mu$ is $F$-quasi-invariant,
    \item\label{it:BQ4} $\mu(A)\le 2\tilde\mu(A)$,
    \item\label{it:BQ5} $\tilde\mu([A]_F)=\mu([A]_F)$ for every Borel set $A\subseteq X$.
\end{enumerate}
\end{proposition}
\begin{proof}
It is a standard fact that $i_n^*\mu$ is a probability Borel measure for every $n<\omega$ and thus $\tilde\mu\in \mathcal{P}(X)$.

Let $A\subseteq X$.
Then since $i_0=\mathrm{id}_X$ we have
$$\frac{\mu(A)}{2}\le \frac{\mu(A)}{2}+\sum_{0<n<\omega}\frac{i^*_n\mu(A)}{2^{n+1}}=\tilde{\mu}(A).$$
This implies Item~\ref{it:BQ4}.
It is easy to see that $i_n^*\mu([A]_F)=\mu([A]_F)$ for every  Borel $A\subseteq X$.
This gives Item~\ref{it:BQ5}.

It remains to show that $\tilde{\mu}$ is $F$-quasi-invariant.
Let $A\subseteq X$ be such that $\tilde{\mu}(A)=0$.
Suppose that $\mu([A]_F)=\tilde{\mu}([A]_F)\not=0$.
There must be some $n<\omega$ such that $\nu(i_n^{-1}(A))\not=0$ because $[A]_F=\bigcup_{n<\infty}i_n^{-1}(A)$.
Then we have 
$$\tilde{\mu}(A)\ge \frac{i_n^*\mu(A)}{2^{n+1}}>0,$$
which is a contradiction.
\end{proof}

\subsection{$\mathcal{H}_c$ and $\mathcal{H}^\leadsto_c$ for $\mathcal{G}$}\label{HcBorel}

Let $\mathcal{G}=(V,E)$ be a Borel multi-graph with bounded maximum degree.
Suppose that $c;E\to [\Delta(\mathcal{G})+\pi(\mathcal{G})]$ is a proper partial Borel colouring.
Note that since each connected component of $\mathcal{G}$ is countable we may use the definition from Subsection~\ref{Estimates} to define the bipartite graphs $\mathcal{H}_c$ and $\mathcal{H}_c^\leadsto$ with Borel partitions $U_c\subseteq E$ and $\dom(c)\subseteq E$.
Clearly we have $\mathcal{H}_c,\mathcal{H}_c^\leadsto\subseteq F_{\mathcal{E}}$ and the estimates on $\deg_{\mathcal{H}_c},\deg_{\mathcal{H}^\leadsto_c}$ from Propositions~\ref{EstimateSimpleVizing}, \ref{EstimateSimpleVizingDown}, \ref{EstimateVizing} and~\ref{EstimateVizingDown} are still valid.

\begin{proposition}\label{HisBorel}
The bipartite graphs $\mathcal{H}_c$ and $\mathcal{H}_c^\leadsto $ are Borel as subsets of $U_c\times \dom(c)$.
\end{proposition}
\begin{proof}
Intuitively this follows from the local nature of the definitions from Section~\ref{sec:CountMult}.
We comment only briefly why these are Borel.
(Also, similar arguments are used in many places in the next sections.)
Write $Y=E^{<\omega}\cup E^\omega$ for the standard Borel space of all finite and countable sequences of edges of $\mathcal{G}$.
Recall that $s_0,s_1:E\to V$ are Borel maps that assign to each edge its vertices.
The following objects are Borel:
\begin{enumerate}[(i),nosep]
	\item\label{it:HB1} $V_c(s_l(\_),\_):U_c\to Y$ where $l\in \{0,1\}$ and $V_c(s_l(e),e)$ is the Vizing chain,
	\item\label{it:HB2} $\mathrm{SPB}_l\subseteq U_c\times E$ where $l\in \{0,1\}$ and $(e,f)\in \mathrm{SPB}_l$ if $f\in V_c(s_l(e),e)$ is superb,
	\item\label{it:HB3} $W_c(s_l(\_),\_\leadsto \_):\mathrm{SPB}_l\to Y$ where $l\in \{0,1\}$ and $W_c(s_l(e),e \leadsto f)$ is the iterated Vizing chain for every $(e,f)\in \mathrm{SPB}_l$.
\end{enumerate}
Once we see this, then we are done.
For example
\begin{eqnarray*}
(e,f)\in \mathcal{H}^\leadsto_c & \Leftrightarrow & (\exists p,q\in \mathbb{N})\ (\exists l\in \{0,1\}) \ (e,V_c(s_l(e),e)(p))\in \mathrm{SPB}_l \\
&&\& \ \ W_c\left(s_l(e),e\leadsto V_c(s_l(e),e)(p)\right)(q)=f,
\end{eqnarray*}
which implies that $\mathcal{H}_c^\leadsto$ is Borel by the Lusin-Novikov Uniformisation Theorem
(\cite[Theorem~18.10]{Kechris:cdst}).
Next we comment why the objects from \ref{it:HB1}, \ref{it:HB2} and \ref{it:HB3} are Borel.

\ref{it:HB1}:
Fix $l\in \{0,1\}$ and let $e\in U_c$.
The assignment $e \mapsto F_c(s_l(e),e)$ depends only on the radius-2 neighbourhood of $e\in \mathcal{G}$ and is therefore Borel. Likewise, the indicator function of $F_c(s_l(e),e)$ being augmenting as well as the pair $j<k$  and the (smallest possible) colours $\alpha,\beta$ returned by respectively Proposition~\ref{TwoIndices} and Proposition~\ref{VizingPath} are Borel functions of~$e$.
If the first
critical index $i$ is not equal to $j$, then the alternating path $P_c(v_j,\alpha/\beta)$ ends in $s_l(e)$; so the set of the corresponding $e$ is the countable union over $m\in \mathbb{N}$, the length of this path, of locally defined and thus Borel sets.
Finally, the assignment $e\mapsto P_c(s_l(e),e)$ is clearly Borel, as it is determined by the above parameters.

\ref{it:HB2}:
Fix $l\in \{0,1\}$ and let $f\in V_c(s_l(e),e)$.
The fact that $f$ is not superb can be seen in a finite  neighbourhood around $f$ in $\mathcal{G}$.
This implies that $\mathrm{SPB}_l$ is Borel.

\ref{it:HB3}:
Fix $l\in \{0,1\}$ and let $(e,f)\in \mathrm{SPB}_l$, which is Borel by (2).
By (1) we know that the alternating colours of $P_c(s_l(e),e)$ can be computed in a Borel way from $e$.
As in~(1), the second critical index and the type of $f$ are Borel.
Thus $(e,f)\mapsto F_c(s_l(e),e\leadsto f)$ and the (smallest possible) alternating colours of $P_c(s_l(e),e\leadsto f)$ are Borel assignments.
This allows us to construct $W_c(s_l(e),e\leadsto f)$ in a Borel fashion and we are done.
\end{proof}

\section{Invariant Measures}\label{sec:Invariant}

Let $\mathcal{G}=(V,E)$ be a Borel multi-graph of bounded maximum degree and $\mu \in \mathcal{P}(V)$.
Recall that for given $\mu\in \mathcal{P}(V)$ we defined
$\chi_\mu(\mathcal{G})$
to be the minimum value of $\chi_B(\mathcal{G}\upharpoonright C)$ where $C\subseteq V$ is a Borel $\mu$-co-null set.
Also, the measurable chromatic index
$\chi_\mu'(\mathcal{G})$ was defined analogously, except we properly colour edges instead of vertices, i.e., $\chi_\mu'(\mathcal{G})$ is the minimum over $\chi_B(\mathcal{E}\upharpoonright C)$ where $C\subseteq E$ is such that $\mu(\{x\in V:(\exists e\in E\setminus C) \ x\in e\})=0$ (see the Introduction).
In this section we prove Theorem~\ref{th:main}, our main result, and its corollary Theorem~\ref{th:k}.
As the first step we show that we can in fact work with measures on $E$ instead of measures on $V$.
This is more convenient for us since all our constructions from Section~\ref{sec:CountMult} are defined on edges rather than on vertices.

\begin{proposition}\label{MeasureOnG}
Let $\mathcal{G}=(V,E)$ be a Borel multi-graph of bounded maximum degree and $\mu\in \mathcal{P}(V)$.
Then $\chi'_\mu(\mathcal{G})\le \chi_{\hat\mu}(\mathcal{E})$.
\end{proposition}
\begin{proof}
Let $\chi_{\hat\mu}(\mathcal{E})=k$.
By the definition there is a Borel $\hat\mu$-co-null set $C\subseteq E$ and a full proper Borel (vertex) colouring $c:C\to [k]$ of $\mathcal{E}$.
We put $D:=\{x\in V:(\exists e\in E\setminus C) \ x\in e\}$.
Using Proposition~\ref{MeasureHat} we have $\mu(D)<\Delta(\mathcal{G})\epsilon$ for every $\epsilon>0$, i.e., $\mu(D)=0$, and the claim follows.
\end{proof}

Using Proposition~\ref{MeasureHat} and Proposition~\ref{MeasureOnG} we see that in order to prove that $\chi'_{\mu}(\mathcal{G})\le \Delta(\mathcal{G})+\pi(\mathcal{G})$ for every $\mathcal{G}$-invariant measure $\mu$ it is enough to show that $\chi_\nu(\mathcal{E})\le \Delta(\mathcal{G})+\pi(\mathcal{G})$ for every $\mathcal{E}$-invariant measure $\nu$.

\begin{proposition}\label{MainEstimate}
Let $\mathcal{G}=(V,E)$ be a Borel multi-graph of bounded maximum degree, $\nu\in \mathcal{I}_{E_\mathcal{E}}$ and $10(\Delta(\mathcal{G})+\pi(\mathcal{G}))^6<L\in \mathbb{N}$.
Suppose that $c;E\to [\Delta(\mathcal{G})+\pi(\mathcal{G})]$ is a partial Borel proper colouring.
\begin{enumerate}[(i), nosep]
    \item\label{it:A} If $c$ cannot be improved in $L$ steps, then
    $$\nu(U_c)\le \frac{(\Delta(\mathcal{G})+\pi(\mathcal{G}))^4}{L}.$$
    \item\label{it:B} If $c$ cannot be iteratively improved in $L$ steps, then
    $$\nu(U_c)\le \frac{(\Delta(\mathcal{G})+\pi(\mathcal{G}))^{15}}{L^2}.$$
\end{enumerate}
\end{proposition}
\begin{proof} In brief, the claimed inequalities follow from the degree bounds given by Propositions~\ref{EstimateSimpleVizing}--\ref{EstimateSimpleVizingDown} and~\ref{EstimateVizing}--\ref{EstimateVizingDown}, and the standard fact that `local double counting inequalities' apply to invariant measures.

Let us give a formal proof. Put $\Delta:=\Delta(\mathcal{G})$ and $\pi:=\pi(\mathcal{G})$.
Let $\mathcal{H}_c$ and $\mathcal{H}_c^\leadsto$ be Borel bipartite graphs assigned to $\mathcal{G}$ with respect to $c$ (see Proposition~\ref{HisBorel}).
Recall that $E(\mathcal{H}_c),E(\mathcal{H}_c^\leadsto)\subseteq (U_c\times \dom(c))\cap F_{\mathcal{E}}$.
By the Feldman-Moore Theorem (\cite[Theorem~1.3]{KechrisMiller:toe}) we find a sequence of Borel injective maps $\{i_n\}_{n<\omega}$ and $\{j_m\}_{m<\omega}$ such that 
\begin{enumerate}[(a)]
    \item $\dom(i_n),\dom(j_m)\subseteq U_c$ for every $n,m\in \mathbb{N}$,
    \item $\operatorname{rng}(i_n), \operatorname{rng}(j_m)\subseteq \dom(c)$ for every $n,m\in \mathbb{N}$,
    \item $(e,i_n(e))\in \mathcal{H}_c$ and $(e',j_m(e'))\in \mathcal{H}_c^\leadsto$ for every $n,m\in \mathbb{N}$ and $e\in \dom(i_n)$, $e'\in \dom(j_m)$,
    \item for every $(e,f)\in \mathcal{H}_c$ and $(e',f')\in \mathcal{H}_c^\leadsto$ there are unique $n,m\in \mathbb{N}$ such that $i_n(e)=f$ and $j_m(e')=f'$.
\end{enumerate}
The assumption that $\nu\in \mathcal{P}(E)$ is $\mathcal{E}$-invariant gives
$$\int_{U_c}\deg_{\mathcal{H}_c}(e)d\nu=\sum_{n<\omega}\nu(\dom(i_n))=$$
$$=\sum_{n<\omega}\nu(\operatorname{rng}(i_n))=\int_{\dom(c)}\deg_{\mathcal{H}_c}(f)d\nu$$
and
$$\int_{U_c}\deg_{\mathcal{H}^\leadsto_c}(e)d\nu=\sum_{m<\omega}\nu(\dom(j_m))=$$
$$=\sum_{m<\omega}\nu(\operatorname{rng}(j_m))=\int_{\dom(c)}\deg_{\mathcal{H}^\leadsto_c}(f)d\nu$$

Now we are ready to derive both parts of the proposition.

\ref{it:A}:
Using Proposition~\ref{EstimateSimpleVizing} and Proposition~\ref{EstimateSimpleVizingDown} we obtain that
$$L\nu(U_c)\le (\Delta+\pi)^4.$$

\ref{it:B}:
Using Proposition~\ref{EstimateVizing} and Proposition~\ref{EstimateVizingDown} we obtain that
$$\left(\frac{L}{2\Delta}\left(\frac{1}{3(\Delta+\pi)^2}\left(\frac{L}{2}-\Delta^5-1\right)-2\Delta^3\right)\right)\nu(U_c)\le (\Delta+\pi)^9.$$
The assumption that $10(\Delta+\pi)^6<L$ together with a simple computation gives the desired estimates.
\end{proof}

Let $\mathcal{G}=(V,E)$ be a Borel multi-graph of bounded maximum degree and put $\Delta:=\Delta(\mathcal{G})$ and $\pi:=\pi(\mathcal{G})$.
Next we describe how to construct,  given an integer $L\in \mathbb{N}$ with $L>2\Delta$ and a proper partial Borel colouring $c;E\to [\Delta+\pi]$, a proper partial Borel colouring $d;E\to [\Delta+\pi]$ such that 
\begin{enumerate}[(i), nosep]
	\item $U_c\supseteq U_d$,
	\item $\nu(\{e\in E:c(e)\not=d(e)\})\le 3L\nu(U_c)$,
	\item $d$ cannot be iteratively improved in $L$ steps. 
\end{enumerate}
We need for this construction $\nu$ to be $\mathcal{E}$-invariant. The construction is fairly standard (see, for example, Elek and Lippner~\cite{ElekLippner10}). Informally speaking, we split all potential augmenting chains of length at most $3L$ into Borel sets, each consisting of vertex-disjoint chains that can be augmented independently of each other; then we take these sets one by one, so that each appears infinitely often, iteratively augmenting all currently possible chains in each set.

Formally, we fix a sequence $\{A_n\}_{n<\omega}$ of Borel subsets of $U_c$ such that each $A_n$ is $6L$-independent in $\mathcal{E}$ and for every $e\in U_c$ there are infinitely many indices $n<\omega$ such that $e\in A_n$.
This is possible by Theorem~\ref{th:KST} since $\mathcal{E}$ has bounded maximum degree. 

We build inductively proper Borel partial colourings $c_n;E\to [\Delta+\pi]$ with the property that $U_{c_n}\supseteq U_{c_{n+1}}$ and then put 
$$d(e):=\lim_{n\to\infty}c_n(e)$$
where the limit is in the discrete finite space $[\Delta+\pi]$, i.e., the limit exists if and only if $\{c_n(e)\}_{n<\omega}$ stabilizes after finitely many steps.
We show that if $e\in \dom(c_m)$ for some $m\in \mathbb{N}$, then $\lim_{n\to\omega}c_n(e)$ exists.
This guarantees that $d$ is a proper Borel partial colouring and $U_c\supseteq U_d$.

Suppose that $c_n$ is defined.
Consider the set $C_{n+1}\subseteq A_{n+1}$ of those $e\in A_{n+1}$ such that there exists $x\in e$ for which $|P_c(x,e)|<L$ or for which there is a superb edge $f\in P_{c_n}(x,e)$ such that $|P_c(x,e)_{i(f)}|\le L$ and $|P_{c_n}(x,e\leadsto f)|\le L$.
The set $C_{n+1}$ is clearly Borel and moreover we can pick such a vertex $x\in e$ and if necessary $f\in P_c(x,e)$ in a Borel way.
Denote them as $x_e$ and if necessary $f_e$.
Define $Q(e)$ to be $V_c(x_e,e)$ in the former and $W_c(x_e,e\leadsto f_e)$ in the latter case.
Then the assignment $e\mapsto Q(e)$ is Borel.
The argument is similar to the one in Proposition~\ref{HisBorel}.
By our assumptions on the independence of $A_{n+1}$ we see that $Q(e_0)$ and $Q(e_1)$ are vertex disjoint for different $e_0,e_1\in C_{n+1}$, i.e., there is no $y\in V$ that is used by both paths simultaneously.
The shift of $c_n$ along all $\{Q(e)\}_{e\in C_{n+1}}$ simultaneously yields a proper Borel partial colouring $c_{n+1}'$.
Because every $Q(e)$ was $c_n$-augmenting we can extend $c_{n+1}'$ even further to a proper Borel partial colouring $c_{n+1}$ such that $U_{c_n}\supseteq U_{c_{n+1}}$.

Note that for every $g\in E$ there are at most $\Delta^{3L}$ possible edges $e\in U_c$ that can cause a modification of a colour of $g$ during our construction.
Because $U_{c_n}\supseteq U_{c_{n+1}}$ we see that the colour of every $g\in E$ eventually stabilizes and therefore $d;E\to [\Delta+\pi]$ is well-defined.

Note also that if $e\in C_{n}$ for some $n<\omega$, then $e\in \dom(c_{l})$ for every $n<l<\omega$ and that the number of colours that were changed when we modified according to $e$ is at most $2\Delta+2L\le 3L$.
The invariance of $\nu$ then implies that the set of edges that changed colour during
our construction has measure at most~$3L\nu(U_c)$.

Let $e\in U_{d}$.
It follows from the construction that its $3L$-neighbourhood stabilizes after finitely many steps of our induction, say $k$ steps.
Let $k<n$ be such that $e\in A_n$.
Then $e\not\in C_n$ since $e\in U_d$.
This implies that $d$ cannot be iteratively improved in $L$ steps.

\begin{theorem}\label{MainInv}
Let $\mathcal{G}=(V,E)$ be a Borel multi-graph of bounded maximum degree and $\nu\in \mathcal{I}_{F_{\mathcal{E}}}$.
Then 
$$\chi_\nu(\mathcal{E})\le \Delta(\mathcal{G})+\pi(\mathcal{G}).$$
\end{theorem}
\begin{proof}
Let us  inductively define a sequence of proper Borel partial colourings $\{d_n\}_{n<\omega}$ such that 
\begin{enumerate}[(i)]
    \item $d_n$ cannot be iteratively improved in $2^n$ steps,
    \item $U_{d_{n}}\supseteq U_{d_{n+1}}$,
    \item $\nu(\{e\in E:d_n(e)\not=d_{n+1}(e)\})<2^{n+3}\nu(U_{d_n})$.
\end{enumerate}
We can build such a sequence inductively as described after Proposition~\ref{MainEstimate}.
We define 
$$d:=\lim_{n\to\infty} d_n.$$
The limit exists for $\nu$-almost every $e\in E$ because by the third inductive property and Proposition~\ref{MainEstimate}
$$\sum_{n<\omega}2^{n+3}\nu(U_{d_n})\le \sum_{n<\omega}2^{n+3}\frac{(\Delta(\mathcal{G})+\pi(\mathcal{G}))^{15}}{2^{2n}}<\infty$$
and by the Borel-Cantelli Lemma the measure of edges that changes the colour infinitely often is $0$.
\end{proof}

\begin{proof}[Proof of Theorem~\ref{th:main}]
Use Theorem~\ref{MainInv} together with the comment before Proposition~\ref{MainEstimate}.
\end{proof}

\begin{proof}[Proof of Theorem~\ref{th:k}.] The upper bound $k'(d)\le d+1$ follows from Theorem~\ref{th:main} and the observation that a Borel matching naturally encodes a Borel involution.

Let us show that $k(d)\le k$, where $k:=\lceil(d+2)/2\rceil$. Let $\GG=(V,\mathcal B,E,\mu)$ be any graphing of maximum degree $d$. Let $E=E_0\cup\ldots\cup E_{d+1}$ be the Borel partition of $E$ returned by  Theorem~\ref{th:main} where $E_0$ consists of uncoloured edges and satisfies $\mu(V(E_0))=0$. By uncolouring each connectivity component containing at least one edge from $E_0$, we can additionally assume by the invariance of the measure $\mu$ that both $V(E_0)$ and $W:=V\setminus V(E_0)$
are \emph{$\mathcal{G}$-invariant} sets (that is, each is equal to its $F_\mathcal{G}$-closure).

 
Consider the Borel graph $\mathcal{G}\upharpoonright W$. 
The proof of Theorem 8.3 in~\cite{CsokaLippnerPikhurko16} shows how to partition its edge set (which is the union of the $d+1$ Borel matchings $E_1,\dots, E_{d+1}$) into $k$ Borel subgraphs whose connectivity components are finite cycles and finite paths, apart from a $\mathcal{G}$-invariant $\mu$-co-null set $U\subseteq W$.
Then it is easy to choose in a Borel way one of the two possible orientations of each finite cycle or finite path, thus obtaining $k$ directed Borel graphs of out-degree at most $1$ on $U$ each naturally encoding a partial Borel bijection.

It remains to find the required functions for $\mathcal{G}\upharpoonright (V(E_0)\cup U)$. Since $V(E_0)\cup U$ has measure 0, we can use the Axiom of Choice to first choose a $(d+1)$-edge-colouring and then orient each path and cycle obtained by grouping all colours into $\lceil(d+1)/2\rceil\le k$ pairs and singletons, with each group corresponding to one function.\end{proof}

\begin{remark}\label{re:lower}\rm For the sake of completeness, let us briefly present the lower bounds.  The bound $k'(d)\ge d+1$ for $d\ge 2$ follows by taking a finite graph $G=(V,E)$ with $\chi'(G)=\Delta(G)+1$ and $\Delta(G)=d$, and considering the graphing $(V,2^V,E,\mu)$ where $\mu$ is the uniform measure on $V$. The bound $k(d)\ge d/2$ follows by considering any graphing $\GG$ with positive measure of vertices having  degree $\Delta(\GG)=d$. For even $d\ge 2$, it can be improved to $d/2+1$ by taking the constructions of Laczkovich \cite{Laczkovich88} for $d=2$ and Conley and Kechris~\cite[Section~6]{ConleyKechris13} for even $d\ge 4$ of a bipartite $d$-regular graphing $\GG$ such
that every Borel matching misses a set of vertices of positive measure. 
\end{remark}

\hide{
The stated upper bound on $k(d)$ can also be derived from a result by Thornton~\cite[Corollary 2.6]{Thornton20arxiv}. In the opposite direction, Theorem~\ref{th:k} gives that the edges of every graphing of maximum degree at most $d$ can be oriented in a measurable way so that each out-degree is at most $\lceil (d+2)/2\rceil$.
}

\section{Approximate edge colouring}\label{sec:QuasiInvariant}

Let $\mathcal{G}=(V,E)$ be a Borel multi-graph of bounded maximum degree and $\mu \in \mathcal{P}(V)$.
Recall that
$\chi_{AP,\mu}(\mathcal{G})$
is the smallest $k\in \mathbb{N}$ such that for every $\epsilon>0$ there is a Borel set $C\subseteq V$ such that $\mu(C)>1-\epsilon$ and $\chi_B(\mathcal{G}\upharpoonright C)\le k$.
We also define
$$\chi_{AP}(\mathcal{G}):=\sup\{\chi_{AP,\mu}(\mathcal{G}):\mu\in P(V)\}.$$
Similarly we define the approximate measurable chromatic indices $\chi'_{AP,\mu}(\mathcal{G})$ and $\chi'_{AP}(\mathcal{G})$ (see the Introduction).

Recall that we write $\mathcal{QI}_{F}$ for the set of all $F$-quasi-invariant measures (where $(X,F)$ is a countable Borel equivalence relation, for example $(V,F_\mathcal{G})$).
It follows from Proposition~\ref{BasicQuasi} that
$$\chi'_{AP}(\mathcal{G})=\sup\{\chi'_{AP,\mu}:\mu\in \mathcal{QI}_{F_\mathcal{G}}\}$$
and similarly for $\chi_{AP}(\mathcal{G})$.
Namely take $\mu\in \mathcal{P}(V)$ and consider $\tilde\mu$.
Suppose that we have $\chi'_{AP,\tilde\mu}(\mathcal{G})=k$.
Then for every $\epsilon>0$ we can find a Borel set $C\subseteq V$ such that $\tilde\mu(C)>1-\frac{\epsilon}{2}$ and $\chi'_B(\mathcal{G}\upharpoonright C)\le k$.
By Proposition~\ref{BasicQuasi} we have $\mu(V\setminus C)\le 2\tilde\mu(V\setminus C)\le \epsilon$.
This implies that $\chi_{AP,\mu}(\mathcal{G})\le k$ and we are done.

As in the invariant case, it is more convenient for us to work with measures on $E$ instead of measures on $X$.

\begin{proposition}\label{pr:AppMeasureOnEdges}
Let $\mathcal{G}=(V,E)$ be a Borel multi-graph of bounded maximum degree and $\mu\in \mathcal{P}(V)$.
Then
$$\chi'_{AP,\mu}(\mathcal{G})\le \chi_{AP,\hat\mu}(\mathcal{E}).$$
\end{proposition}
\begin{proof}
Let $k\in \mathbb{N}$ be such that $\chi_{AP,\hat\mu}(\mathcal{E})=k$.
Pick $\epsilon>0$.
By definition there is $C\subseteq E$ such that $\hat\mu(C)<\epsilon$ and a full proper Borel colouring $c:E\setminus C\to [k]$ of $\mathcal{E}\upharpoonright (E\setminus C)$.
Put $D:=\{x\in V:(\exists e\in C) \ x\in e\}$.
Using Proposition~\ref{MeasureHat} we have $\mu(D)<\Delta(\mathcal{G})\epsilon$.
Note that $E\cap [V\setminus D]^2 \subseteq E\setminus C$.
It follows that we can restrict $c$ to get a full proper Borel edge colouring $d:E\cap [V\setminus D]^2\to [k]$.
This implies that $\chi'_{AP,\mu}(\mathcal{G})\le k$.
\end{proof}

It follows from previous observations that in order to show that 
$$\chi'_{AP}(\mathcal{G})\le \Delta(\mathcal{G})+\pi(\mathcal{G})$$
it is enough to prove the following statement.

\begin{theorem}\label{QuasiVizing}
Let $\mathcal{G}=(V,E)$ be a Borel multi-graph of bounded maximum degree and $\nu\in \mathcal{QI}_{F_{\mathcal{E}}}$.
Then 
$$\chi_{AP,\nu}(\mathcal{E})\le \Delta(\mathcal{G})+\pi(\mathcal{G}).$$
\end{theorem}

We need to introduce a fundamental tool for dealing with quasi-invariant measures.
Let $(X,F)$ be a countable Borel equivalence relation and $\mu\in \mathcal{Q}\mathcal{I}_{F}$.
A \emph{cocycle} for $\mu$ is a Borel function $\rho_\mu:F\to (0,\infty)$ such that 
$$\mu(\phi(A))=\int_{A}\rho_\mu(\phi(x),x)d\mu$$
for every Borel $A\subseteq X$ and a Borel injection $\phi:A\to X$ such that $(\phi(x),x)\in F$ for every $x\in A$.
It exists for every quasi-invariant Borel measure (see, for example,~\cite[Proposition~8.3]{KechrisMiller:toe}).
We think of $\rho_\mu(x,y)$ as the ratio of the ``density'' at $x$ to that at $y$.
Since we only work with one fixed measure we omit the subscript and write simply $\rho$ for the corresponding cocycle.

First we need to adapt Definition~\ref{SimpleImprovement} to  quasi-invariant measures.

\begin{definition}\label{QuasiImprov}
Let $c;E\to [\Delta(\mathcal{G})+\pi(\mathcal{G})]$ be a proper partial colouring, $\nu\in \mathcal{QI}_{F_{\mathcal{E}}}$ and $\rho$ the corresponding cocycle.
We say that $c$ does not \emph{admit improvement of weight $L\in \mathbb{N}$ (with respect to $\nu$)} if for $\nu$-almost every $e\in U_c$ it holds that
$$\sum_{f\in V_c(x,e)\setminus\{e\}}\rho(f,e)\ge L,\quad\mbox{for every $x\in e$.}
$$

\end{definition}

Note that this modification allows to make the same estimates as in the invariant case.
More concretely, in the case when $\nu\in \mathcal{P}(E)$ is $\mathcal{E}$-invariant measure our aim is to compare the measure of some edges that are not coloured with the measure of edges that are members of the corresponding Vizing chains.
For example, if $A\subseteq U_c$ is Borel and $w:A\to V$ is a Borel map such that $w(e)\in e$, the Vizing chains $V_c(w(e),e)$ are pairwise vertex disjoint and $L\le |V_c(w(e),e)\setminus\{e\}|$ for each $e\in A$, then we have simply by invariance
$$L\nu(A)\le \int_A|V_c(w(e),e)\setminus\{e\}|d\nu =\nu(\left\{f\in \dom(c):(\exists e\in A) \ f\in V_c(w(e),e)\right\}).$$
The Definition~\ref{QuasiImprov} allows to compute the same estimates under the assumption that $\nu\in \mathcal{P}(E)$ is $\mathcal{E}$-quasi-invariant.
Namely suppose that $c;E\to [\Delta(\mathcal{G})+\pi(\mathcal{G})]$ does not admit improvement of weight $L\in \mathbb{N}$, $A\subseteq U_c$ and $w:A\to V$ are as above.
Then we have
$$L\nu(A)\le \int_A \left(\sum_{f\in V_c(w(e),e)\setminus\{e\}}\rho(f,e) \right)d\nu=\nu(\left\{f\in \dom(c):(\exists e\in A) \ f\in V_c(w(e),e)\right\}).$$

\begin{proposition}\label{NoImprovementQuasi}
Let $\mathcal{G}=(V,E)$ be a Borel multi-graph of bounded maximum degree, $L\in \mathbb{N}$ and $\nu\in \mathcal{Q}\mathcal{I}_{F_{\mathcal{E}}}$.
Then there is a proper Borel partial colouring $c;E\to [\Delta(\mathcal{G})+\pi(\mathcal{G})]$ that does not admit improvement of weight $L$ (with respect to $\nu$).
\end{proposition}
\begin{proof}
Let $d;E\to [\Delta(\mathcal{G})+\pi(\mathcal{G})]$ be a proper partial Borel colouring.
We define $A_d$ to be the set of edges in $U_d$ such that there is $x\in e$ such that 
$$\sum_{f\in V_d(x,e)\setminus\{e\}}\rho(f,e)<L.$$
It is clear that $A_d$ is a Borel set and that $d$ satisfies the requirements of the proposition if and only if $\nu(A_d)=0$.

We use induction to build a transfinite sequence of proper partial Borel colourings $\{c_{\alpha}\}_{\alpha<\omega_1}$ that satisfy the following
\begin{enumerate}[(i), nosep]
    \item\label{it:1} $c_0=\emptyset$,
    \item\label{it:2} for every $\alpha\le \beta<\omega_1$, we have $\dom(c_\alpha)\subseteq \dom(c_\beta)$ up to a $\nu$-null set,
    \item\label{it:3}  for every $\alpha<\beta<\omega_1$, if $c_\alpha\not=c_\beta$ on a non-$\nu$-null set then $\nu(U_{c_\alpha})>\nu(U_{c_\beta})$,
    \item\label{it:4} for ever $\alpha<\omega_1$, if $\nu(A_{c_\alpha})>0$, then $\dom(c_{\alpha})\not=\dom(c_{\alpha+1})$ on a non-$\nu$-null set.
\end{enumerate}
Note that if we manage to find such a sequence, then we are done.
This follows from the fact that there are no strictly decreasing sequences of real numbers of length $\omega_1$ and therefore there is $\gamma<\omega_1$ such that $\nu(U_{c_\beta})=\nu(U_{c_{\gamma}})$ for every $\beta\ge \gamma$.
This implies that $\nu(A_{c_{\gamma}})=0$ by \ref{it:2}, \ref{it:3} and~\ref{it:4}.

As usual we distinguish the successor and limit step in our construction and we start with the successor.
Suppose that $c_\alpha$ is defined.
We want to find $c_{\alpha+1}$.
Suppose that $\nu(A_{c_\alpha})>0$, otherwise we are done, and denote as $w(e)$ the vertex in $e$ that is a witness to the fact that $e\in A_{c_\alpha}$.
Such a selection $w:A_{c_\alpha}\to V$ can be chosen in a Borel way (the argument is similar as in Proposition~\ref{HisBorel}).
Define
$$Y_\alpha:=\{e\in A_{c_\alpha}:|V_{c_{\alpha}}(w(e),e)|=\infty\},$$
$$Z_\alpha:=\{e\in A_{c_\alpha}:|V_{c_\alpha}(w(e),e)|<\infty\}.$$
Clearly we have $A_{c_\alpha}=Y_\alpha\cup Z_\alpha$ and both sets are Borel.

Suppose that $\nu(Z_\alpha)>0$.
In this case we may find $k<\omega$ such that $\nu(Z_{\alpha,k})>0$ where
$$Z_{\alpha,k}:=\{e\in A_{c_\alpha}:|V_{c_\alpha}(w(e),e)|=k\}.$$
The desired modification follows similar lines as the algorithm before Theorem~\ref{MainInv}.
Namely pick $Z\subseteq Z_{\alpha,k}$ that is $6k$-independent in $\mathcal{E}$ such that $\nu(Z)>0$ and do the shifts of $c_{\alpha}$ along $V_{c_\alpha}(w(e),e)$ simultaneously for every $e\in Z$ to define some proper partial Borel colouring~$c'_{\alpha+1}$.
Because each $V_{c_{\alpha}}(w(e),e)$ is $c_\alpha$-augmenting we can extend $c'_{\alpha+1}$ to the desired proper partial Borel colouring $c_{\alpha+1}$ that satisfies all the required properties.
We define $B_\alpha:=Z$ and $C_\alpha:=\{e\in E:c_\alpha(e)\not=c_{\alpha+1}(e)\}\cup Z$.
It follows from the definition of $A_{c_\alpha}$ and the cocycle relation that 
$$\nu(C_\alpha)\le (L+1)\nu(B_{\alpha})$$
and
$$\nu(U_{c_{\alpha+1}})<\nu(U_{c_\alpha}).$$

Suppose that $\nu(Z_\alpha)=0$.
Then we must have $\nu(Y_\alpha)>0$.
For colours $\gamma,\delta$ we define $Y_{\alpha,\gamma/\delta}$ to be the set of those $e\in Y_\alpha$ such that $P_c(w(e),e)$ is an infinite $\gamma/\delta$-path.
It is clear that $Y_{\alpha,\gamma/\delta}$ is Borel.
It also follows that there is a pair of colours $\gamma/\delta$ such that $\nu(Y_{\alpha,\gamma/\delta})>0$.
Pick some $Y\subseteq Y_{\alpha,\gamma/\delta}$ that is $5$-independent in the graph $\mathcal{E}$ and still $\nu(Y)>0$.
It follows from the definition of $V_{c_\alpha}(w(e),e)$ that if $e_0\not=e_1\in Y$, then no $f_0\in V_{c_\alpha}(w(e_0),e_0)$ and $f_1\in V_{c_\alpha}(w(e_1),e_1)$ share vertex.
This is because $e_0,e_1$ are far apart and both $P_c(w(e_0),e_0)$ and $P_c(w(e_1),e_1)$ are infinite $\delta/\gamma$-paths.
This implies that we can make the shift of $c_\alpha$ along each $e\in Y$ simultaneously and define $c_{\alpha+1}$ to be the shift.
Since each Vizing's chain is infinite, $c_{\alpha+1}$ satisfies all the required properties.
We put $B_\alpha:=Y$ and $C_\alpha:=\{e\in G:c_\alpha(e)\not=c_{\alpha+1}(e)\}\cup Y$.
It follows from the definition of $A_{c_\alpha}$ and the cocycle relation that 
$$\nu(C_\alpha)\le (L+1)\nu(B_{\alpha})$$
and
$$\nu(U_{c_{\alpha+1}})<\nu(U_{c_\alpha}).$$
This finishes the construction in the successor stage.

Suppose that $\alpha$ is a limit ordinal and $\{c_\beta\}_{\beta<\alpha}$ is defined.
Write $B_\beta=\dom(c_{\beta+1})\setminus \dom(c_\beta)$, i.e., the set of edges that get coloured in the $\beta$-th step.
We have clearly 
$$\sum_{\beta<\alpha}\nu(B_\beta)\le 1.$$
Define $C_\beta$ to be the set of edges that changed the colour in the $\beta$-th step, i.e., $f\in C_\beta$ if $f\in B_\beta$ or $c_{b+1}(f)\not=c_\beta(f)$.
We have by construction of the successor stage (see above) that 
$$\sum_{\beta<\alpha}\nu(C_\beta)\le L+1.$$
Since $\alpha<\omega_1$ there is a bijection $q:\mathbb{N}\to\alpha$.
By the Borel-Cantelli Lemma we have
$$\nu\Big(\bigcap_{n<\omega}\bigcup_{k\ge n}C_{q(k)}\Big)=0.$$
This implies that up to $\nu$-null set every $f\in \bigcup_{\beta<\alpha}\dom(c_\beta)$ changes its colour only finitely many times.
In other words we can define 
$$c_\alpha(f):=\lim_{\beta\to\alpha}c_{\beta}(f)$$
up to $\nu$-null subset of $f\in \bigcup_{\beta<\alpha} \dom(c_\beta)$.
\end{proof}

\begin{proof}[Proof of Theorem~\ref{QuasiVizing} and of Theorem~\ref{th:ap}]
Let $\mu \in \mathcal{P}(V)$.
The note before Proposition~\ref{pr:AppMeasureOnEdges} implies that we may assume that $\mu\in \mathcal{Q}\mathcal{I}_{F_\mathcal{G}}$ and Proposition~\ref{pr:AppMeasureOnEdges} implies that the $F_{\mathcal{E}}$-quasi-invariant measure $\hat{\mu}$ satisfies 
$$\chi'_{AP,\mu}(\mathcal{G})\le \chi_{AP,\hat\mu}(I_\mathcal{G}).$$
This shows that Theorem~\ref{QuasiVizing} implies Theorem~\ref{th:ap}.

We show Theorem~\ref{QuasiVizing}.
Let $c;E\to [\Delta(\mathcal{G})+\pi(\mathcal{G})]$ be a proper partial Borel colouring that does not admit improvement of weight $L$ (with respect to $\nu$).
Such a colouring exists according to Proposition~\ref{NoImprovementQuasi}.
Consider the bipartite Borel graph $\mathcal{H}_c$.
Recall that $\mathcal{H}_c\subseteq F_{\mathcal{E}}$.
The properties of the cocycle $\rho$ then give that
$$\int_{U_c}\left(\sum_{f\in V_c(v_0(e),e)\setminus\{e\}\cup V_c(v_1(e),e)\setminus\{e\}}\rho(f,e)\right)d\nu=\int_{\dom(c)}\deg_{\mathcal{H}_c}(g)d\nu.$$
Using Proposition~\ref{EstimateSimpleVizing} and the fact that $c$ does not admit improvement of weight $L$, we conclude that
$$L\nu(U_c)\le (\Delta(\mathcal{G})+\pi(\mathcal{G}))^4.$$
Since $L$ can be arbitrary we are done.
\end{proof}

\section*{Acknowledgements}

The authors thank the anonymous referee for useful comments.

\bibliography{bibexport}

\begin{bibdiv}
\begin{biblist}

\bib{Abert10questions}{unpublished}{
      author={Ab{\'e}rt, M.},
       title={Some questions},
        date={2010},
        note={\texttt{http://www.renyi.hu/\symbol{126}abert/questions.pdf}},
}

\bib{Bernshteyn19am}{article}{
      author={Bernshteyn, A.},
       title={Measurable versions of the {L}ov{\'a}sz {L}ocal {L}emma and
  measurable graph colorings},
        date={2019},
     journal={Adv.\ Math.},
      volume={353},
       pages={153\ndash 223},
}

\bib{Bernshteyn20arxiv}{unpublished}{
      author={Bernshteyn, A.},
       title={Distributed algorithms, the {Lov\'asz} {L}ocal {L}emma, and
  descriptive combinatorics},
        date={2020},
        note={E-print arxiv:2004.04905v3},
}

\bib{Bernshteyn20arxiv2}{unpublished}{
      author={Bernshteyn, A.},
       title={A fast distributed algorithm for {$(\Delta+1)$}-edge-coloring},
        date={2020},
        note={E-print arxiv:2006.157003},
}

\bib{Brooks41}{article}{
      author={Brooks, R.~L.},
       title={On colouring the nodes of a network},
        date={1941},
     journal={Proc. Cambridge Philos. Soc.},
      volume={37},
       pages={194\ndash 197},
}

\bib{ConleyKechris13}{article}{
      author={Conley, C.~T.},
      author={Kechris, A.~S.},
       title={Measurable chromatic and independence numbers for ergodic graphs
  and group actions},
        date={2013},
     journal={Groups Geom. Dyn.},
      volume={7},
       pages={127\ndash 180},
}

\bib{ConleyMarksTuckerdrob16}{article}{
      author={Conley, C.~T.},
      author={Marks, A.~S.},
      author={Tucker-Drob, R.~D.},
       title={Brooks' theorem for measurable colorings},
        date={2016},
     journal={Forum of Math., Sigma},
      volume={4},
       pages={e16, 23},
}

\bib{CsokaLippnerPikhurko16}{article}{
      author={Cs\'oka, E.},
      author={Lippner, G.},
      author={Pikhurko, O.},
       title={{K\H onig}'s line coloring and {V}izing's theorems for
  graphings},
        date={2016},
     journal={Forum of Math., Sigma},
      volume={4},
       pages={40pp.},
}

\bib{ElekLippner10}{article}{
      author={Elek, G.},
      author={Lippner, G.},
       title={Borel oracles. {A}n analytical approach to constant-time
  algorithms},
        date={2010},
     journal={Proc.\ Amer.\ Math.\ Soc.},
      volume={138},
       pages={2939\ndash 2947},
}

\bib{Furman11}{incollection}{
      author={Furman, A.},
       title={A survey of measured group theory},
        date={2011},
   booktitle={Geometry, rigidity, and group actions},
      series={Chicago Lectures in Math.},
   publisher={Univ. Chicago Press, Chicago, IL},
       pages={296\ndash 374},
}

\bib{Gaboriau02}{incollection}{
      author={Gaboriau, D.},
       title={On orbit equivalence of measure preserving actions},
        date={2002},
   booktitle={Rigidity in dynamics and geometry ({C}ambridge, 2000)},
   publisher={Springer, Berlin},
       pages={167\ndash 186},
}

\bib{Gaboriau10}{inproceedings}{
      author={Gaboriau, D.},
       title={Orbit equivalence and measured group theory},
        date={2010},
   booktitle={Proceedings of the {I}nternational {C}ongress of
  {M}athematicians. {V}olume {III}},
   publisher={Hindustan Book Agency, New Delhi},
       pages={1501\ndash 1527},
}

\bib{Gupta66}{article}{
      author={Gupta, R.~P.},
       title={The chromatic index and the degree of a graph},
        date={1966},
     journal={Notices Amer.\ Math.\ Soc.},
      volume={13},
       pages={719},
}

\bib{Kechris:cdst}{book}{
      author={Kechris, A.~S.},
       title={Classical descriptive set theory},
      series={Graduate Texts in Mathematics},
   publisher={Springer-Verlag, New York},
        date={1995},
      volume={156},
}

\bib{Kechris:gaega}{book}{
      author={Kechris, A.~S.},
       title={Global aspects of ergodic group actions},
      series={Mathematical Surveys and Monographs},
   publisher={American Mathematical Society, Providence, RI},
        date={2010},
      volume={160},
}

\bib{KechrisMarks:survey}{unpublished}{
      author={Kechris, A.~S.},
      author={Marks, A.~S.},
       title={Descriptive graph combinatorics},
        date={2016},
        note={Manuscript, 104pp.},
}

\bib{KechrisMiller:toe}{book}{
      author={Kechris, A.~S.},
      author={Miller, B.~D.},
       title={Topics in orbit equivalence},
      series={Lecture Notes in Mathematics},
   publisher={Springer, Berlin},
        date={2004},
      volume={1852},
}

\bib{KechrisSoleckiTodorcevic99}{article}{
      author={Kechris, A.~S.},
      author={Solecki, S.},
      author={Todorcevic, S.},
       title={Borel chromatic numbers},
        date={1999},
     journal={Adv.\ Math.},
      volume={141},
       pages={1\ndash 44},
}

\bib{Konig16}{article}{
      author={K{\H{o}}nig, D.},
       title={Gr{\'a}ok {\'e}s alkalmaz{\'a}suk a determin{\'a}nsok {{\v Z}s} a
  halmazok elm{\'e}le{\'e}re},
        date={1916},
     journal={Matematikai {\'e}s Term{\'e}szettudom{\'a}nyi {\'Ertesít\H o}},
      volume={34},
       pages={104\ndash 119},
}

\bib{Laczkovich88}{article}{
      author={Laczkovich, M.},
       title={Closed sets without measurable matching},
        date={1988},
     journal={Proc.\ Amer.\ Math.\ Soc.},
      volume={103},
       pages={894\ndash 896},
}

\bib{Lovasz:lngl}{book}{
      author={{Lov\'asz}, L.},
       title={Large networks and graph limits},
      series={Colloquium Publications},
   publisher={Amer.\ Math.\ Soc.},
        date={2012},
}

\bib{Marks16}{article}{
      author={Marks, A.~S.},
       title={A determinacy approach to {Borel} combinatorics},
        date={2016},
     journal={J.\ Amer.\ Math.\ Soc.},
      volume={29},
       pages={579\ndash 600},
}

\bib{ScheideStiebitz12}{article}{
      author={Scheide, D.},
      author={Stiebitz, M.},
       title={The maximum chromatic index of multigraphs with given {$\Delta$}
  and {$\mu$}},
        date={2012},
     journal={Graphs Combin.},
      volume={28},
       pages={717\ndash 722},
}

\bib{Shalom05}{incollection}{
      author={Shalom, Y.},
       title={Measurable group theory},
        date={2005},
   booktitle={European {C}ongress of {M}athematics},
   publisher={Eur. Math. Soc., Z\"urich},
       pages={391\ndash 423},
}

\bib{Shannon49}{article}{
      author={Shannon, C.~E.},
       title={A theorem on coloring the lines of a network},
        date={1949},
     journal={J. Math. Physics},
      volume={28},
       pages={148\ndash 151},
}

\bib{StiebitzScheideToftFavrholdt:gec}{book}{
      author={Stiebitz, M.},
      author={Scheide, D.},
      author={Toft, B.},
      author={Favrholdt, L.~M.},
       title={Graph edge coloring},
      series={Wiley Series in Discrete Mathematics and Optimization},
   publisher={John Wiley \& Sons, Inc., Hoboken, NJ},
        date={2012},
}

\bib{Thornton20arxiv}{unpublished}{
      author={Thornton, R.},
       title={Orienting {B}orel graphs},
        date={2020},
        note={E-print arxiv:2001.01319},
}

\bib{Vizing64}{article}{
      author={Vizing, V.~G.},
       title={On an estimate of the chromatic class of a {$p$}-graph},
        date={1964},
     journal={Diskret. Analiz No.},
      volume={3},
       pages={25\ndash 30},
}

\end{biblist}
\end{bibdiv}


\end{document}